\documentclass[a4paper,11pt]{article}

\usepackage{amssymb,amsmath,amsthm,xcolor,pdfsync,enumerate,comment,hyperref}
\usepackage[margin=1in]{geometry}
\allowdisplaybreaks

\newtheorem{thrm}{Theorem}[section]
\newtheorem{cor}[thrm]{Corollary}
\newtheorem{lem}[thrm]{Lemma}
\newtheorem{prop}[thrm]{Proposition}

\theoremstyle{definition}
\newtheorem{defn}[thrm]{Definition}
\newtheorem{exm}[thrm]{Example}

\theoremstyle{definition}
\newtheorem{rem}[thrm]{Remark}

\renewcommand{\iff}{\Leftrightarrow}

\newcommand{\pMod}[1]%
{%
	{\sf pMod}{(#1)}%
}

\newcommand{\Mod}[1]%
{%
	{\sf Mod}{(#1)}%
}

\newcommand{\ESMod}[1]%
{%
	{\sf ESMod}{(#1)}%
}

\newcommand{\ESModst}[1]%
{%
	{\sf ESMod^*}{(#1)}%
}

\newcommand{\InvESMod}[1]%
{%
	{\sf InvESMod}{(#1)}%
}

\newcommand{\ESSet}[1]%
{%
	{\sf ESSet}{(#1)}%
}

\newcommand{\A}[1]%
{%
	{\sf A}{(#1)}%
}

\newcommand{\Ab}%
{%
	{\sf Ab}%
}

\newcommand{\impl}%
{%
	\Rightarrow
}

\newcommand{\id}{\mathrm{id}}

\newcommand{\cI}[1]%
{%
	\mathcal I{(#1)}%
}

\newcommand{\cIui}[1]%
{%
	{\mathcal I}_{ui}{(#1)}%
}

\newcommand{\cS}[1]%
{%
	\mathcal S{(#1)}%
}

\newcommand{\cD}%
{%
	\mathcal D%
}

\newcommand{\cU}[1]%
{%
	\mathcal U{(#1)}%
}

\newcommand{\bbN}%
{%
	\mathbb N%
}

\newcommand{\bbZ}%
{%
	\mathbb Z%
}

\newcommand{\dom}[1]%
{%
	\operatorname{\mathrm{dom}}{#1}%
}

\newcommand{\ran}[1]%
{%
	\operatorname{\mathrm{ran}}{#1}%
}

\newcommand{\End}[1]%
{%
	\operatorname{\mathrm{End}}{#1}%
}

\DeclareMathOperator{\im}{im}%

\newcommand{\Hom}%
{%
	\mathrm{Hom}%
}

\newcommand{\Ext}%
{%
	\mathrm{Ext}%
}

\newcommand{\m}{^{-1}}

\newcommand{\0}{\theta}
\newcommand{\e}{\varepsilon}
\newcommand{\G}{\Gamma}

\newcommand{\bbC}{{\mathbb C}}

\begin{document}

\begin{center}
{\bf\Large Partial cohomology of groups}
\end{center}
 \vspace {3 mm}

\vspace {5 mm}
\centerline {\bf M. Dokuchaev\footnote{Partially supported by CNPq of Brazil} and  M. Khrypchenko\footnote{Supported by FAPESP of Brazil (process number: 2012/01554--7)}}

 \vspace {5 mm}

{\small
\centerline{Instituto de Matem\'atica e Estat\'istica --- Universidade de S\~ao Paulo,}
\centerline{Caixa Postal 66281, S\~ao Paulo, SP}
\centerline{05315--970 --- Brazil}
 \vspace {5 mm}

\begin{abstract}
 We develop a cohomology theory of groups based on partial actions and explore its  relation with the partial Schur multiplier as well as with cohomology of inverse semigroups.  
\end{abstract}
}



\section*{Introduction}\label{sec-intro}

In~\cite{E-1}, \cite{E0} and \cite{E1} R.~Exel initiated a  successful method to study   $C^*$-algebras,  which is based on new concepts, namely, those of a partial action, the corresponding crossed product  and a partial representation, as well as on the interaction between them.  Relevant classes of
$C^*$-algebras have been  shown to have the structure of a non-trivial  crossed product by a partial action, permitting  one to investigate their internal structure, representations  and  K-theory (see \cite{E-2}, \cite{E-3}, \cite{ELQ}, \cite{QR} and \cite{AraEKa}).

The first algebraic results on the above mentioned notions were established in~\cite{E1}, \cite{DEP}, \cite{S}, \cite{KL} and \cite{DE}, which together with the development of the partial  Galois theory in~\cite{DFP} stimulated an intensive algebraic activity on partial actions, corresponding crossed products and partial representations (see the surveys~\cite{D} and~\cite{F2}). In particular, applications were obtained to graded algebras in~\cite{DE} and \cite{DES1}, to Hecke algebras in~\cite{E3} and to Leavitt path algebras in~\cite{GR}. 

A purely ring theo\-re\-tic version  of the general notion of a  twisted partial group action, initially introduced in the context of   $C^*$-algebras  in~\cite{E0},   was given in~\cite{DES1}.  It  allowed one to show that any group graded algebra, satisfying some reasonably mild restrictions, is stably isomorphic (in a certain sense) to a  crossed product by a twisted partial action, establishing thus an algebraic counterpart of a result from~\cite{E0}. The concept  involves a general twisting which satisfies the $2$-cocycle identity in some restricted sense, and one naturally wonders  what kind of cohomology theory would fit this. It is complicated to give such a theory which would embrace the full generality of~\cite{DES1}, since the twisting takes its values in  multiplier algebras of products of some ideals indexed by group elements. 
Nevertheless, imposing a reasonable restriction on the partial action, namely, that it is unital,  makes it possible to define partial cohomology. Actually, this restriction is assumed in almost all algebraic papers around partial actions, because it  provides an appropriate technical framework.

One feels more confident in developing such a cohomology theory if one shows that it matches some other concepts in a way how the usual group cohomology does. In particular, the Schur multiplier of a group $G$ over the complex numbers $\bbC $ is isomorphic to $H^2(G, \bbC ^{\ast})$ with  trivial action of  $G$ on  $ \bbC ^{\ast}.$  Thus  a theory of partial projective group representations could be a testing field for our cohomology. Such a theory was developed in~\cite{DN}, \cite{DN2} and \cite{DoNoPi}, in particular the structure of the partial Schur multiplier was explored.  Further results on the latter topic were obtained in~\cite{NP} and \cite{Pi}.

A twisted partial action of a group $G$ on a commutative ring $A$ falls into two parts: a partial action $\theta$ of $G$ on $A$ and its twisting. With an additional assumption that $\theta$ is unital (which means that the domains involved in $\theta$ are generated by central idempotents), one can derive  the concept of a partial $2$-cocycle (the twisting) whose  values belong to groups of invertible elements of appropriate ideals of $A.$ The concept of a partial $2$-coboundary then follows from that of an equivalence of twisted partial actions introduced in~\cite{DES2}. Replacing $A$  by a commutative multiplicative monoid, one comes to the definition of the second co\-ho\-mo\-lo\-gy group $H^2(G,A).$ Thus instead of a usual $G$-module we deal with a  partial $G$-module, which is a commutative monoid $A$ with a unital partial action $\theta$ of $G$ on $A$. The groups $H^n(G,A)$ with arbitrary $n$ are defined in a similar way (see Section~\ref{sec-notions}).  Replacing $A$ by an appropriate  submonoid one may actually assume that $A$ is inverse (see Remark~\ref{rem-semigr_of_U(prod_D_x)}). Next one asks how to obtain these groups using, say, projective resolutions. It turns out that the category of partial $G$-modules is not abelian since some sets of morphisms may be empty.   Fortunately, our cohomology can be related to Lausch-Leech-Loganathan cohomology of inverse semigroups (see~\cite{Lausch}, \cite{Leech75} and~\cite{Loganathan81}) via the inverse monoid $\cS G$ introduced by R. Exel in~\cite{E1} to serve partial actions and partial representations of $G$ (later in~\cite{KL} J.~Kellendonk and M.~V.~Lawson showed, using a result by M.~Szendrei~\cite{Szendrei89}, that $\cS G$ is isomorphic to the Birget-Rhodes expansion~\cite{Birget-Rhodes84,Birget-Rhodes89} of $G$). From a unital partial action of $G$ on $A$ one comes to an action of $\cS G$ and then to an ``almost'' Lausch's $\cS G$-module structure on $A$. The latter can be seen as a module in the sense of H.~Lausch over an epimorphic image of $\cS G$, provided that $A$ is an inverse partial $G$-module (see Definition~\ref{defn-inverse_part_mod}). Thus our category is made up of abelian ``components'' which are categories of Lausch's modules over epimorphic images of $\cS G$ (see Remark~\ref{rem-pi-comp_is_isom_to_StrMod(S')}). This way we are able to define free objects and free resolutions which lead to $H^n(G,A)$.

It turns out (Theorem~\ref{thrm-max-generated}) that the epimorphic images $S$ of $\cS G$, which appear in our treatment, are exactly the max-generated $F$-inverse monoids (see~\cite[p. 196]{Lawson2002}), whose maximum group image is $G$. Thus, our (non-abelian) category of inverse partial $G$-modules englobes all categories of $S$-modules (in the sense of H.~Lausch) for such $S$. Furthermore, given an inverse semigroup $S$ with maximum group image $G$ and a (global) $G$-module $A$, M.~Loganathan~\cite{Loganathan81} constructed a specific $S$-module, which we denote by $\hat A$, such that the cohomology of $S$ with values in $\hat A$ coincides with the usual cohomology of $G$ with values in $A$, establishing this way a certain relation between the cohomology of an inverse semigroup and the cohomology of its maximum group image. For the class of max-generated $F$-inverse monoids our approach gives a more comprehensive connection between the Lausch-Leech-Loganathan cohomology of $S$ and the partial cohomology of $G$.

The  article is organized as follows. After establishing  the basic concepts in Section~\ref{sec-notions}, we show in Section~\ref{Schur}  that the partial Schur multiplier $pM(G)$  is a union of $2$-cohomology groups with values in some,  in general non-trivial,  partial $G$-modules (Theorem~\ref{pM(G)_is_covered_by_R(theta^Gamma,S^Gamma)}). The passage from partial $G$-modules to  $\cS G$-modules via unital actions of $\cS G$ is done in Section~\ref{sect-modules_over_S(G)} (Theorem~\ref{thrm-from_part_iso_to_endo}). Our notion of a module $A$ over an inverse semigroup $S$ is more general than that introduced in~\cite{Lausch} (Definition~\ref{defn-S-mod}), and the relation with modules in the sense of~\cite{Lausch}  (strict $S$-modules) is given in Proposition~\ref{prop-from_part_act_to_Lausch}. 
Here  a crucial role is played by the semi-direct product $A\rtimes S$ (Definition~\ref{defn-cross_prod_for_S-mod}), which is the   semigroup analogue of the crossed product introduced in~\cite{EV}. If $S=\cS G$ then $A\rtimes S$ is isomorphic to the crossed product $A*_\0 G$ (Remark~\ref{rem-A*G_is_isom_to_A*S(G)}) defined in~\cite{DN}. 
In the same section we also define free $S$-modules and show their existence in Proposition~\ref{prop-free_epi-strict_S-mod} using Lausch's construction  of (strict) free $S$-modules. Next, in Section~\ref{sec-from_part_cohom_to_Lausch_cohom} we give a free resolution in order to obtain the cohomology groups $H^n (G, A)$ (Theorem~\ref{thrm-R_n_if_free_res_of_Z_S}). Finally, in Section~\ref{sec-max-generated} we describe the cohomology of a max-generated $F$-inverse monoid in terms of the partial cohomology of its maximum group homomorphic image (Corollary~\ref{cor-H^n_S(A)-as-H^n(G,A)}). We conclude the section by giving an interpretation in the setting of the partial group cohomology of a related result by M.~Loganathan~\cite[Proposition 3.6]{Loganathan81}.

There is a very natural point to be clarified, namely, the application of low dimensional partial group cohomology to extensions.  The constraints of the present article do not allow us to treat it here, and it is being considered in a separate paper under elaboration. 

	\section{The notions}\label{sec-notions}

		Let $G$ be a group and $A$ a semigroup. Recall from~\cite{DN} that a partial action $\theta$ of $G$ on $A$ is a collection of semigroup isomorphisms $\theta_x:A_{x^{-1}}\to A_x$, where $A_x$ is an ideal of $A$, $x\in G$, such that
	\begin{enumerate}[(i)]
		\item $A_1=A$ and $\theta_1=\id_A$; 
		\item $\theta_x(A_{x^{-1}}\cap A_y)=A_x\cap A_{xy}$;
		\item $\theta_x\circ\theta_y=\theta_{xy}$ on $A_{y^{-1}}\cap A_{y^{-1}x^{-1}}$.
	\end{enumerate}
We consider the case when $A$ is a commutative monoid and each ideal $A_x$ is unital, i.\,e. $A_x$ is generated by an  idempotent $1_x =1_x^A,$ which is central in $A.$ In this situation we shall say that $\theta $ is a {\it unital} partial action. Then $A_x\cap A_y=A_xA_y$, so the properties (ii) and (iii) from the above definition can be replaced by
	\begin{enumerate}[(i')]
		\setcounter{enumi}{1}
		\item $\theta_x(A_{x^{-1}}A_y)=A_xA_{xy}$;
		\item $\theta_x\circ\theta_y=\theta_{xy}$ on $A_{y^{-1}}A_{y^{-1}x^{-1}}$.
	\end{enumerate}
Note also that (ii') implies a more general equality 
	\begin{equation}\label{eq-theta_of_product_of_ideals}
		\theta_x(A_{x^{-1}}A_{y_1}\dots A_{y_n})=A_xA_{xy_1}\dots A_{xy_n},
	\end{equation}
which easily follows by  observing that $A_{x^{-1}}A_{y_1}\dots A_{y_n}=(A_{x^{-1}}A_{y_1})\dots (A_{x^{-1}}A_{y_n})$.

	\begin{defn}\label{defn-partial_module}
		A commutative monoid $A$ with a unital  partial action $\theta$ of $G$ on $A$  will be called {\it a (unital) partial $G$-module}. 
	\end{defn}
Recall from~\cite{DN2} that a morphism of partial actions $(A,\theta) \to (A',\0')$ of $G$  is a homomorphism of semigroups $\varphi:A\to A'$ such that $\varphi(A_x)\subseteq A'_x$ and $\varphi\circ\theta_x=\0'_x\circ\varphi$ on $A_{x^{-1}}$.

	\begin{defn}\label{defn-morph_of_part_mod}
		A  morphism of (unital) partial  $G$-modules $\varphi:(A,\theta)\to(A',\0')$ is a morphism of partial actions such that its restriction on each $A_x$ is a homomorphism of monoids $A_x\to A'_x$.
	\end{defn}

The category of (unital) partial $G$-modules and their morphisms will be denoted by $\pMod G$.  Sometimes $(A,\0)$ will be simplified to  $A$.

	\begin{defn}\label{defn-cochain}
		Let $A\in\pMod G$ and $n$ be a positive integer. {\it An $n$-cochain} of $G$ with  values in $A$ is a function $f:G^n\to A$, such that $f(x_1,\dots,x_n)$ is an invertible element of the ideal $A_{(x_1,\dots,x_n)}=A_{x_1}A_{x_1x_2}\dots A_{x_1\dots x_n}$. By {\it a $0$-cochain} we shall mean an invertible element of $A$. 
	\end{defn}

Denote the set of $n$-cochains by $C^n(G,A)$. It is an abelian group under the pointwise multiplication. Indeed, its identity is 
\[
	e_n(x_1,\dots,x_n)=1_{x_1}1_{x_1x_2}\dots 1_{x_1\dots x_n}
\]
and the inverse of $f\in C^n(G,A)$ is $f^{-1}(x_1,\dots,x_n)=f(x_1,\dots,x_n)^{-1}$, where $f(x_1,\dots,x_n)^{-1}$ means the inverse of $f(x_1,\dots,x_n)$ in $A_{(x_1,\dots,x_n)}$.

	\begin{defn}\label{defn-coboundary_hom}
		Let $(A,\theta)\in\pMod G$ and $n$ be a positive integer. For any $f\in C^n(G,A)$ and $x_1,\dots,x_{n+1}\in G$ define
	\begin{align}\label{eq-coboundary_hom}
		(\delta^nf)(x_1,\dots,x_{n+1})&=\theta_{x_1}(1_{x_1^{-1}}f(x_2,\dots,x_{n+1}))\notag\\
&\prod_{i=1}^nf(x_1,\dots , x_ix_{i+1}, \dots,x_{n+1})^{(-1)^i}\notag\\
&f(x_1,\dots,x_n)^{(-1)^{n+1}}. 
	\end{align}
Here the inverse elements are taken in the corresponding ideals. If $n=0$ and $a$ is an invertible element of $A$, we set $(\delta^0a)(x)=\theta_x(1_{x^{-1}}a)a^{-1}$.
	\end{defn}

	\begin{prop}\label{prop-coboundary_hom}
		The map $\delta^n$ is a homomorphism $C^n(G,A)\to C^{n+1}(G,A)$, such that 
	\begin{equation}\label{eq-deltasquare}
		\delta^{n+1}\delta^nf=e_{n+2}
	\end{equation} 
for any $f\in C^n(G,A)$.
	\end{prop}
	\begin{proof}
		Let $f\in C^n(G,A)$. We check first  that $\delta^nf\in C^{n+1}(G,A)$. Indeed, for $x_1,\dots,x_{n+1}\in G$ the element $f(x_2,\dots,x_{n+1})$ is invertible in $A_{(x_2,\dots,x_{n+1})}$. Then, being multiplied by $1_{x_1^{-1}}$, it becomes an invertible element of $A_{x_1^{-1}}A_{(x_2,\dots,x_{n+1})}$. Therefore, $\theta_{x_1}(1_{x_1^{-1}}f(x_2,\dots,x_{n+1}))$ is in\-ver\-tible in $A_{(x_1,\dots,x_{n+1})}$, because $\theta_{x_1}$ maps isomorphically $A_{x_1^{-1}}A_{(x_2,\dots,x_{n+1})}$ onto $A_{(x_1,\dots,x_{n+1})}$ by~\eqref{eq-theta_of_product_of_ideals}. Since the pro\-duct of invertible elements of some ideals is invertible in the product of these ideals, then by~\eqref{eq-coboundary_hom} the image $(\delta^nf)(x_1,\dots,x_{n+1})$ is invertible in
	\[
		A_{(x_1,\dots,x_{n+1})}\left(\prod_{i=1}^nA_{(x_1,\dots, x_ix_{i+1}, \dots,x_{n+1})}\right) A_{(x_1,\dots,x_n)}= A_{(x_1,\dots,x_{n+1})}.
	\]

As $A$ is commutative, to see that $\delta^n$ is a homomorphism, it  suffices  to note that
	\begin{align*}
		\theta_{x_1}(1_{x_1^{-1}}fg(x_2,\dots,x_{n+1}))&=\theta_{x_1}(1_{x_1^{-1}}f(x_2,\dots,x_{n+1})1_{x_1^{-1}}g(x_2,\dots,x_{n+1}))\\
		&=\theta_{x_1}(1_{x_1^{-1}}f(x_2,\dots,x_{n+1}))\theta_{x_1}(1_{x_1^{-1}}g(x_2,\dots,x_{n+1})).
	\end{align*}

It remains to prove~\eqref{eq-deltasquare}.  Take arbitrary  $x_1,\dots,x_{n+2}\in G.$  The factors in the product  $(\delta^{n+1}\delta^nf)(x_1,\dots,x_{n+2})$ to which the partial action is applied  are as follows:
	\begin{align*}
		&\theta_{x_1}(1_{x_1^{-1}}\theta_{x_2}(1_{x_2^{-1}}f(x_3,\dots,x_{n+2}))),\\ 
		&\theta_{x_1x_2}(1_{x_2^{-1}x_1^{-1}}f(x_3,\dots,x_{n+2})^{-1}),\\
		&\theta_{x_1}(1_{x_1^{-1}}f(x_2,\dots,x_{n+1})^{(-1)^{n+1}}),\\
		&\theta_{x_1}(1_{x_1^{-1}}f(x_2,\dots,x_{n+1})^{(-1)^{n+2}}),\\
		&\theta_{x_1}(1_{x_1^{-1}}f(x_2,\dots,x_ix_{i+1},\dots,x_{n+2})^{(-1)^{i-1}}),2\le i\le n+1,\\ 
		&\theta_{x_1}(1_{x_1^{-1}}f(x_2,\dots,x_ix_{i+1},\dots,x_{n+2})^{(-1)^i}),2\le i\le n+1.
	\end{align*}
The product of all the factors, except the first two, is $e_{n+2}(x_1,\dots,x_{n+2})$ for $n\geq 1.$ For $n=0$ the product is $e_1(x_1).$  Furthermore,
	\begin{align*}
		&\theta_{x_1}(1_{x_1^{-1}}\theta_{x_2}(1_{x_2^{-1}}f(x_3,\dots,x_{n+2})))\\
&=\theta_{x_1}(\theta_{x_2}(1_{x_2^{-1}}1_{x_2^{-1}x_1^{-1}}f(x_3,\dots,x_{n+2})))\\ 
&=\theta_{x_1x_2}(1_{x_2^{-1}}1_{x_2^{-1}x_1^{-1}}f(x_3,\dots,x_{n+2})) 
	\end{align*}
by the property (iii') from the definition of a partial action. After multiplying this by the second factor we shall obtain
	\begin{align*}
		&\theta_{x_1x_2}(1_{x_2^{-1}}1_{x_2^{-1}x_1^{-1}}e_n(x_3,\dots,x_{n+2}))\\ 
		&=1_{x_1}1_{x_1x_2}e_n(x_1x_2x_3,x_4,\dots,x_{n+2})\\
		&=e_{n+2}(x_1,\dots,x_{n+2}).
	\end{align*}

Any other factor in  $(\delta^{n+1}\delta^nf)(x_1,\dots,x_{n+2})$ appears together with its inverse,  as in the classical case, and multiplying such a pair we obtain a product of some of   the idempotents $1_{x_1}, 1_{x_1x_2},\dots.$ Thus, $(\delta^{n+1}\delta^nf)(x_1,\dots,x_{n+2})=e_{n+2}(x_1,\dots,x_{n+2})$ as desired. 
	\end{proof}

	\begin{defn}\label{defn-cohomology}
		The map $\delta^n$ is called {\it a coboundary homomorphism}. As in the classical case we define the abelian groups $Z^n(G,A)=\ker{\delta^n}$, $B^n(G,A)=\im{\delta^{n-1}}$ and $H^n(G,A)=\ker{\delta^n}/\im{\delta^{n-1}}$ of {\it partial $n$-cocycles, $n$-co\-boun\-da\-ries and $n$-cohomologies} of $G$ with values in $A$, $n\ge 1$ $(H^0(G,A)=Z^0(G,A)=\ker{\delta^0})$.
	\end{defn}

For example, 
	\begin{align*}
		H^0(G,A)&=Z^0(G,A)=\{a\in{\mathcal U}(A)\mid{\theta}_x(1_{x^{-1}}a) = 1_xa,  \forall x \in G \},\\
		B^1(G,A)&=\{ f\in C^1(G,A) \mid f(x) = {\theta}_x(1_{x^{-1}}a)  a^{-1}, \, \mbox{for some}\, a \in  {\mathcal U}(A)  \} 
	\end{align*} 
(here and below $\cU A$ denotes the group of invertible elements of $A$). Notice that $H^0(G,A)$ is the subgroup of $\theta $-invariants of  $\cU A$ (see \cite[p. 79]{DFP}). Furthermore, 
	\[
		( \delta^{1} f) (x,y) = \theta _x ( 1_{x^{-1}}f(y) ) f(xy)^{-1}f(x)
	\]  
for $f \in  C^1 (G,A),$ so that 
	\begin{align*}
		Z^1(G,A) &=\{f\in C^1 (G,A)\mid 1_xf(xy)=f(x) \, \theta _x ( 1_{x^{-1}}f(y) ), \forall x,y\in G \},\\
		B^2 (G,A) &=\{ g \in  C^2 (G,A) \mid g(x,y) = \theta _x ( 1_{x^{-1}}f(y) ) f(xy)^{-1}f(x)\\
		&\phantom{=\{}\mbox{ for some } f \in  C^1 (G,A) \}.
	\end{align*}
For $n=2$ we have 
	\[
		( \delta^{2} f) (x,y,z) = \theta _x ( 1_{x^{-1}}f(y,z) )  \; f(xy,z)^{-1} \; f(x,yz) \; f(x,y)^{-1},
	\] 
with $f \in C^2 (G,A),$ and 
	\begin{align*}
		Z^2(G,A) &=\{f\in C^2 (G,A) \mid \theta _x ( 1_{x^{-1}}f(y,z) ) \; f(x,yz) = f(xy,z)  \, f(x,y),\\
		&\phantom{=\{} \forall x,y,z \in G \}.
	\end{align*}

Observe that if one takes a  unital twisted partial action (see~\cite[Def. 2.1]{DES1}) of  $G$ on a commutative ring $A$ then it is readily seen that the twisting is a $2$-cocycle with values in the partial $G$-module $A,$ and the concept of  equivalent  unital twisted partial actions from~\cite[Def. 6.1]{DES2} is exactly the notion of cohomologous $2$-cocycles  from Definition~\ref{defn-cohomology}.

	\begin{prop}\label{prop-H^n_is_a_functor}
		The map which sends a partial $G$-module $A$ to the sequence
	\[
		C^0(G,A)\overset{\delta^0}{\to} C^1(G,A)\overset{\delta^1}{\to}\dots\overset{\delta^{n-1}}{\to} C^n(G,A)\overset{\delta^n}{\to}\dots
	\]
is a functor from $\pMod G$ to the category of complexes of abelian groups.
	\end{prop}
	\begin{proof}
		Let $\varphi:(A,\theta)\to(A',\0')$ be a morphism of partial $G$-modules and $f\in C^n(G,A)$. Define $\widetilde\varphi_nf=\varphi\circ f:G^n\to A'$. We need to show that $\widetilde\varphi_nf\in C^n(G,A')$. Indeed, for any $x_1,\dots, x_n$ the element $f(x_1,\dots,x_n)$ is invertible in $A_{(x_1,\dots,x_n)}$. Since 
	\[
		\varphi(1^A_{x_1}1^A_{x_1x_2}\dots 1^A_{x_1\dots x_n})=1^{A'}_{x_1}1^{A'}_{x_1x_2}\dots 1^{A'}_{x_1\dots x_n},
	\]
then the restriction of $\varphi$ to $A_{(x_1,\dots,x_n)}$ is a homomorphism of monoids $A_{(x_1,\dots,x_n)}\to A'_{(x_1,\dots,x_n)}.$  Therefore, $\varphi(f(x_1,\dots,x_n))$ is invertible in $A'_{(x_1,\dots,x_n)}$ and its inverse is $\varphi(f(x_1,\dots,x_n)^{-1})$. Consequently, in order to  prove that $\{\widetilde\varphi_n\}_{n\ge 0}$ is a morphism of complexes, i.\,e.   to verify that $\widetilde\varphi_{n+1}\circ\delta^n=\delta^n\circ\widetilde\varphi_n,$ we  only need to show that
	\[
		\varphi\circ\0_{x_1}(1^A_{x_1^{-1}}f(x_2,\dots,x_{n+1}))=
\0'_{x_1}(1^{A'}_{x_1^{-1}}\varphi\circ f(x_2,\dots,x_{n+1})).
	\]
 By the definition of a morphism of partial $G$-modules
	\begin{align*}
		\varphi\circ\theta_{x_1}(1^A_{x_1\m}f(x_2,\dots,x_{n+1}))
		&=\0'_{x_1}\circ\varphi(1^A_{x_1\m}f(x_2,\dots,x_{n+1}))\\
		&=\0'_{x_1}(\varphi(1^A_{x_1\m})\varphi(f(x_2,\dots,x_{n+1})))\\
		&=\0'_{x_1}(1^{A'}_{x_1\m}\varphi(f(x_2,\dots,x_{n+1}))),
	\end{align*}
as desired.

The remaining properties of a functor are straightforward.
	\end{proof}

Composing the above functor with the homology functor we have the next.

	\begin{cor}\label{cor-H^n_is_a_funct}
		For any $n\ge 0$ the map $A\mapsto H^n(G,A)$ determines a functor from $\pMod G$ to the category $\Ab$ of abelian groups.
	\end{cor}

%%%%%%%%%%%%%%%%%%%
%%%%%%%%%%%%%%%%%
%%%%%%%%%%%%%%%%%%%%

	\section{\texorpdfstring{$H^2(G,A)$}{H\texttwosuperior(G,A)} and the partial Schur multiplier}\label{Schur}

Let $K$ be a field. Recall from~\cite[Definition 2]{DN} that a {\it $K$-semigroup} is a semigroup $S$ with zero, endowed with a map $K\times S\to S$, satisfying $(\alpha\beta)s=\alpha(\beta s)$, $\alpha(st)=(\alpha s)t=s(\alpha t)$, $1_Ks=s$, $0_Ks=0_S$ for all $\alpha,\beta\in K$ and $s,t\in S$. One says that $S$ is {\it $K$-cancellative}, if $\alpha s=\beta s$ implies $\alpha=\beta$ for arbitrary $\alpha,\beta\in K$ and nonzero $s\in S$.

Let $M$ be a $K$-cancellative monoid. We recall one of the equivalent definitions of a {\it partial projective representation} of $G$ in $M$ given in~\cite[Theorem 3]{DN}. It is a map $\Pi:G\to M$, such that 
    \begin{align*}
    \Pi(x\m)\Pi(xy)\ne 0_M\iff \Pi(x)\Pi(y)\ne 0_M\iff \Pi(xy)\Pi(y\m)\ne 0_M,
    \end{align*}
    and there exists a unique partially defined map $\sigma:G^2\to K^*$ (called the {\it factor set} of $\Pi$) with $\dom\sigma=\{(x,y)\in G^2\mid \Pi(x)\Pi(y)\ne 0_M\}$ and 
    \begin{enumerate}
        \item $\Pi(x\m)\Pi(x)\Pi(y)=\Pi(x\m)\Pi(xy)\sigma(x,y)$;
        \item $\Pi(x)\Pi(y)\Pi(y\m)=\Pi(xy)\Pi(y\m)\sigma(x,y)$
    \end{enumerate}
    for all $x,y\in\dom\sigma$. As in~\cite{DN}, we set $\sigma(x,y)=0_K$ whenever $(x,y)\not\in\dom\sigma$.

We know from~\cite[Theorem 6]{DN} that each partial projective representation $\Pi:G\to M$ with factor set $\sigma$ induces $(A^\Pi,\theta^\Pi)\in\pMod G$. Here $A^\Pi$ is the commutative submonoid of $M$ generated by $\alpha \epsilon_x$, where $\alpha\in K$ and $\epsilon_x$ is the (possibly zero) idempotent $\Pi(x)\Pi(x^{-1})\sigma(x^{-1},x)^{-1}$. The ideal $A^\Pi_x$ is $\epsilon_xA^\Pi$ and $\theta^\Pi_x(a)=\Pi(x)a\Pi(x^{-1})\sigma(x^{-1},x)^{-1}$ for $a\in A^\Pi_{x^{-1}}$.

	\begin{rem}\label{rem-H^n(G,theta^Gamma,S^Gamma)_for_global_Gamma}
		If $\Pi:G\to M$ is a (global) projective representation, then $(A^\Pi,\theta^\Pi)$ is isomorphic to the trivial $G$-module $K$ (where $K$ is considered as a multiplicative semigroup). Moreover,  $H^n(G,A^\Pi)\cong H^n(G,K^*)$, where $H^n(G,K^*)$ is the classical $n$-th cohomology group of $G$ with values in the trivial $G$-module $K^*$. 
	\end{rem}
\noindent Indeed, in the global case $\Pi(x)\Pi(x\m)=\sigma(x,x\m)1_M$ and, since by (21) from~\cite{DN} $\sigma(x,x^{-1}) =  \sigma(x^{-1}, x)$, then $\epsilon_x=1_M$, so $A^\Pi_x=A^\Pi=K\cdot 1_M$ and $\theta^\Pi_x=\id_{K\cdot 1_M}$ for all $x\in G$. Any $f\in C^n(G,A^\Pi)$ is thus identified with a map $G^n\to K^*$, i.\,e. with a classical $n$-cochain, and~\eqref{eq-coboundary_hom} becomes the classical coboundary homomorphism under the trivial $G$-action on $K^*$.
	\begin{defn}\label{defn-K-linear-part-G-mod}
		By {\it a $K$-linear} partial $G$-module we shall mean $(A,\0)\in\pMod G$ such that $A$ is $K$-cancellative and each  $\theta _x : A_{x^{-1}} \to A_x$ is a  $K$-map.
	\end{defn}
Evidently, for any partial projective representation $\Pi$ of $G$ in a $K$-cancellative monoid $M$,  $\theta^\Pi$ gives a structure  of a  $K$-linear partial $G$-module on $A^\Pi .$

	\begin{defn}\label{defn-adjusted}
		Following~\cite{DN2} a $K$-linear partial $G$-module $A$  will be called {\it adjusted} if $A$ is generated by $\alpha 1_x$ ($\alpha\in K$, $x\in G$).
	\end{defn} 
Observe that in this case $E(A)=\langle 1_x\mid x\in G\rangle$ and $A=K\cdot E(A)$. Clearly, for any partial projective representation $\Pi : G \to M$ the $K$-linear partial  $G$-module $(A^\Pi , \theta^\Pi )$ is adjusted.  We recall from~\cite{DN} the next.

	\begin{defn}\label{defn-twisted} 
		Let $A$ be a  $K$-linear partial $G$-module. A $K$-valued {\it twisting related to $A$} is  a  function $\sigma :G \times G \to K$ satisfying the following conditions:
	\begin{enumerate}[(i)]
		\item $\sigma (x,y) =0_K \Leftrightarrow A_xA_{xy} =
0_A\;\; (x,y \in G);$
		\item $\sigma (x,1_G) = \sigma (1_G,x) =1_K$ for all $x \in G$ such that $A_x \neq 0_A;$
		\item $A_xA_{xy}A_{xyz} \ne 0_A \Rightarrow
\sigma (x,y) \sigma (xy, z) =  \sigma (y,z) \sigma (x, yz)  $ with
$ x,y,z \in G.$
	\end{enumerate}
	\end{defn} 

		By (iv) of Theorem~3 from~\cite{DN2} for any adjusted $A\in\pMod G$ and any twisting $\sigma$ related to $A$ there is a partial projective representation $\Pi:G\to M$ with factor set $\sigma$ such that $A$ is isomorphic to $A^\Pi$. Thus, among the  partial $G$-modules, the adjusted ones are precisely those, which come from partial projective representations of $G$.

	\begin{lem}\label{lem-inv_el_of_S^Gamma_x}
		Let $A$ be  an adjusted ($K$-linear) partial $G$-module and $0_A\ne e\in E(A)$. Then for any $a\in\cU{eA}$ there exists a unique $\alpha\in K^*$, such that $a=\alpha e$. 
	\end{lem}
	\begin{proof}
	    Let $a=\alpha ef\in\cU{eA}$, where $\alpha\in K$, $f\in E(A)$. Note that $\alpha\ne 0_K$, since otherwise $a=0_A\not\in\cU{eA}$. So, $\alpha\in K^*$. It follows from $\alpha efa^{-1}=e$ that $ef=e$ and hence $a=\alpha e$. The uniqueness of $\alpha$ is explained by the $K$-cancellative pro\-per\-ty of $A$.
	\end{proof}

Given $A\in \pMod G,$ a partial $2$-cocycle $f \in  Z^2(G,A)$ will be called {\it normalized} if $f(1_G,1_G)=1_A.$ Then using  $(\delta^2f)(x,1_G,1_G)=(\delta^2f)(1_G,1_G,x)=1_x$ we readily see that 
	\begin{equation}\label{eq-normalized}
		f(1_G,x) = f(x,1_G)=1_x \;\;\; \forall x \in G.
	\end{equation} 
The subgroup of  $Z^2(G,A)$ formed by the  normalized partial $2$-cocycles will be denoted by $NZ^2(G,A).$ 

	\begin{rem}\label{rem-each_cocycle_is_cohom_to_normalized}
		For each $f\in Z^2(G,A)$ there is $\widetilde f\in NZ^2(G,A)$ which is {\it cohomologous} to $f$, i.\,e. $f=\widetilde f\cdot\delta^1 g$ for some $g\in C^1(G,A)$.
	\end{rem} 
\noindent For one can take  $\widetilde f (x,y) = f(x,y) \0 _x (1_{x\m} f(1_G,1_G))^{-1}$ and  $g(x)=f(1_G,1_G)1_x\in\mathcal U(A_x)$ for all $x, y \in G$.

Recall from~\cite{DN} that the  factor sets of all partial projective representations of $G$ form a commutative inverse monoid $pm(G)$ under pointwise multiplication. By~\cite[Theorem IV.2.1]{Howie} the semigroup $pm(G)$ is a semilattice of (abelian) groups $\bigsqcup_{e\in E(pm(G))}pm(G)_e$, where $pm(G)_e$ is the subgroup of elements which are invertible with respect to $e$. Clearly, idempotents of $pm(G)$ are precisely those factor sets $\sigma$ which take values $0_K$ or $1_K$. They are obviously identified with their domains $\dom\sigma=\{(x,y)\in G^2\mid\sigma(x,y)=1_K\}$. Thus, $pm(G)=\bigsqcup pm(G)_D$, where $D$ runs over the domains of factor sets of all partial projective representations and $pm(G)_D$ denotes the subgroup of factor sets with domain $D$. The quotient semigroup $pM(G)=pm(G)/{\sim}$, where
	\[
		\sigma\sim\tau\Leftrightarrow\sigma(x,y)=\eta(x)\eta(xy)^{-1}\eta(y)\tau(x,y)
	\]
for some function $\eta:G\to K^*$, is called the partial Schur multiplier of $G$. It has the similar representation $pM(G)=\bigsqcup pM(G)_D$ as a semilattice of abelian groups (see~\cite[Theorem 5]{DN} for more details).

	\begin{prop}\label{prop-Z^2_is_embedded_into_pm}
		Let $(A,\0)$ be  an adjusted  $K$-linear partial $G$-module. Then  $NZ^2(G,A)$ is isomorphic to the subgroup of $pm(G)$ consisting of the  factor sets of all partial projective representations $\Pi:G\to M$, such that $(A^\Pi,\0^\Pi)$ is isomorphic to $(A,\0)$.
	\end{prop}
	\begin{proof}
		Let $f\in NZ^2(G,A)$. Define $\sigma_f:G\times G\to K$ by
	\begin{equation}\label{eq-def_of_sigma_f}
		\sigma_f(x,y)=
	\begin{cases}
		\alpha_{xy}, & \mbox{if $A_x A_{xy}\ne 0_A$ and $f(x,y)=\alpha_{xy}1_x 1_{xy}$},\\
		0_K, & \mbox{if $A_x A_{xy}=0_A$}.
	\end{cases}
	\end{equation}
In particular, $f(x,y)=\sigma_f(x,y) 1_x 1_{xy}$. By Lemma~\ref{lem-inv_el_of_S^Gamma_x} the map $\sigma_f$ is well-defined. Note that $\sigma_f$ is a $K$-valued twisting related to $A$.  Indeed, (i) of Definition~\ref{defn-twisted} immediately follows from~\eqref{eq-def_of_sigma_f}, whereas   (ii) is a consequence of~\eqref{eq-normalized}. Next, suppose that $A_x A_{xy} A_{xyz}\ne 0_A$, i.\,e. $1_x 1_{xy} 1_{xyz}\ne 0_A$. 
Using the $2$-cocycle equality $(\delta^2f)(x,y,z)=1_x 1_{xy} 1_{xyz}$, the fact that $\theta _x: A_{x^{-1}}\to A_x$ is a morphism of $K$-monoids and the $K$-cancellative property, we obtain
	\[
		\sigma_f(y,z)\sigma_f(xy,z)^{-1}\sigma_f(x,yz)\sigma_f(x,y)^{-1}=1_K,
	\]
so (iii) of Definition~\ref{defn-twisted} is also true.

		Having the adjusted $K$-linear partial $G$-module $A$ and the twisting $\sigma_f$ related to it, by Theorem~8 from~\cite{DN} one can construct a partial projective representation $\Pi:G\to M$ whose factor set is $\sigma_f$. Using $\Pi$ we obtain the $K$-linear partial $G$-module $A^\Pi$. Since $A$ is generated by $\alpha 1_x$ as a semigroup ($\alpha\in K$, $x\in G$), Corollary~11 from~\cite{DN} (or (iv) of Theorem 3 from~\cite{DN2})  implies that $A^\Pi$ is isomorphic to $A$.

		Conversely, let $\Pi:G\to M$ be a partial projective representation with  factor set $\tau$, such that $A^\Pi$ is isomorphic to $A$. Define $g_\tau:G\times G\to A$ by
	\begin{equation}\label{eq-def_of_f_tau}
		g_\tau(x,y)=\tau(x,y)1_x 1_{xy}.
	\end{equation}
Note that by Theorem~6 from~\cite{DN} $\tau$ is a twisting related to $A^\Pi$, so $\tau(x,y)=0_K\Leftrightarrow A^\Pi_xA^\Pi_{xy}=0_{A^\Pi}$ which is equivalent to $A_x A_{xy}=0_A$, because $A^\Pi$ is isomorphic to $A$. Therefore, $g_\tau(x,y)$ is an invertible element of $A_x A_{xy}$ and thus $g_\tau\in C^2(G,A)$. We show next  that 
	\begin{equation}\label{eq-cocycle_identity_in_Z^2}
		(\delta^2g_{\tau })(x,y,z)=1_x 1_{xy} 1_{xyz}.
	\end{equation} 
If $A_xA_{xy}A_{xyz}\ne 0_A$, then $A^\Pi_xA^\Pi_{xy}A^\Pi_{xyz}\ne 0_{A^\Pi}$, and~\eqref{eq-cocycle_identity_in_Z^2} follows from the $2$-cocycle identity  for $\tau$ (see  (iii) of Definition~\ref{defn-twisted}). If $A_xA_{xy}A_{xyz}=0_A$, then $1_x1_{xy}1_{xyz}=0_A$, so both sides of~\eqref{eq-cocycle_identity_in_Z^2} are zero. Thus, $g_\tau\in Z^2(G,A)$. Moreover, $g_\tau\in NZ^2(G,A)$ because $\tau(1_G,1_G)=1_K$ by (ii) of Definition~\ref{defn-twisted}.

In order to check   that the map $NZ^2(G,A) \ni f\mapsto\sigma_f \in pm(G)$ is injective, it is enough to notice that by~\eqref{eq-def_of_sigma_f} and~\eqref{eq-def_of_f_tau} one has that $g_{{\sigma}_f}=f.$

Clearly $\sigma_{fg}=\sigma_f\sigma_g$, so $f\mapsto\sigma_f$ is a semigroup monomorphism of $NZ^2(G,A)$ into $pm(G),$ by means of which  $NZ^2(G,A)$ can be identified with  a subgroup of $pm(G),$ as desired. 
	\end{proof}

	\begin{rem}\label{rem-hom_cocycles_induce_equiv_factor_sets}
		Given two normalized partial 2-cocycles $f$ and $g$, it is readily seen that $f$ is cohomologous to $g$ if and only if $\sigma_f\sim\sigma_g$.
	\end{rem}

	\begin{cor}\label{cor-H^2_is_embedded_into_pM}
		The embedding $f\mapsto\sigma_f$ of $NZ^2(G,A)$ into $pm(G)$ induces the embedding  of the group $H^2(G,A)$ into $pM(G)$.
	\end{cor}
\noindent This follows from Remarks~\ref{rem-each_cocycle_is_cohom_to_normalized} and~\ref{rem-hom_cocycles_induce_equiv_factor_sets}. 

Given an adjusted  $K$-linear partial $G$-module $A$, denote by   $R(G,A)$ the subgroup of $pM(G)$ formed by the equivalence classes of $K$-valued twistings related to $A$.

	\begin{rem}\label{rem-description_of_R(theta^Gamma,S^Gamma)}
		The  image of $H^2(G,A)$ in $pM(G)$ coincides with  $R(G,A).$
	\end{rem}
\noindent Indeed, this immediately follows from the proof of Proposition~\ref{prop-Z^2_is_embedded_into_pm} together with the fact that if $\tau $ is a twisting related to $\theta $ and ${\tau }'$ is a factor set equivalent to $\tau$, then  ${\tau }'$ is also a twisting related to $\theta .$ The latter is a direct consequence  of Definition~\ref{defn-twisted}.

	\begin{rem}\label{rem-R(theta^Gamma,S^Gamma)_lies_in_one_component}
		If an element of $R(G,A)$ belongs to some component $pM_D(G)$ of $pM(G)$, then $R(G,A)\subseteq pM_D(G)$.
	\end{rem}
\noindent For by (i) of Definition~\ref{defn-twisted} all twistings related to $A$ have the same domain $\{(x,y)\in G^2\mid A_x A_{xy}\ne 0_A\}$. 

	\begin{defn}\label{defn-tw_equivalence}
		Let $A$ and $A'$ be  $K$-linear partial $G$-modules. We shall say that $A$ and $A'$ are {\it twisting-equivalent} if the sets of $K$-valued twistings related to $A$ and $A'$ coincide.
	\end{defn} 

From  Definition~\ref{defn-twisted} we easily obtain the next sufficient condition of twisting-equivalence.
	
	 \begin{rem}
	 Given two $K$-linear partial $G$-modules $A$ and $A'$, if for all $x,y,z \in G$
	  $$
	  A_xA_{xy}A_{xyz} = 0_A   \Leftrightarrow  A'_x A'_{xy} A'_{xyz} = 0_{A'},
	  $$
	  then $A$ and $A'$ are  twisting-equivalent.
	 \end{rem}
		
The relation between the partial Schur multiplier and the partial cohomology is summarized in the next.

	\begin{thrm}\label{pM(G)_is_covered_by_R(theta^Gamma,S^Gamma)}
		Each component $pM_D(G)$ of $pM(G)$  is the union of the subgroups $R(G,A)\cong H^2(G,A),$ where $A$ runs over the representatives  of twisting-equivalence classes of  adjusted $K$-linear  partial $G$-modules such that 
	\[
		\{(x,y)\in G^2\mid A_xA_{xy}\ne 0_A\}=D.
	\]
	\end{thrm}
	\begin{proof}
		Let $\sigma\in pm_D(G)$ be a factor set of a partial projective representation $\Pi:G\to M$. Then  $\sigma $ is a $K$-valued twisting related to $A^\Pi$, so its equivalence class lies in  $ R(G,A^\Pi).$  Hence $R(G,A^\Pi)\subseteq pM_D(G)$ by Remark~\ref{rem-R(theta^Gamma,S^Gamma)_lies_in_one_component}. The partial $G$-module $A^\Pi$ is adjusted by its de\-fi\-ni\-tion and $\{(x,y)\in G^2\mid A^\Pi_xA^\Pi_{xy}\ne 0_{A^\Pi}\}=\dom\sigma=D.$  
Finally,  $R(G,A^\Pi)\cong H^2(G,A^\Pi),$ thanks to Corollary~\ref{cor-H^2_is_embedded_into_pM}   and 
Remark~\ref{rem-description_of_R(theta^Gamma,S^Gamma)}.
	\end{proof}
	
%%%%%%%%%%%%%
%%%%%%%%%%%%%
%%%%%%%%%%%%%%%

\section{Partial modules and modules over Exel's monoid}\label{sect-modules_over_S(G)}

Recall that a semigroup $S$ is called {\it inverse} if for any $s\in S$ there is a unique $s^{-1}\in S$ (called {\it the inverse of $s$}) such that $ss^{-1}s=s$ and $s^{-1}ss^{-1}=s^{-1}$. A natural example of an inverse semigroup is the monoid $\cI X$ of all partial bijections (i.\,e. bijections between subsets) of a set $X$. The product $\varphi\psi$ in $\cI X$ is the composition $\varphi\circ\psi$ defined for all $x\in X$, such that $\varphi(\psi(x))$ makes sense. More precisely, $\varphi\psi$ is the bijection
	\[
		\varphi\circ\psi:\psi^{-1}(\ran\psi\cap\dom\varphi)\to\varphi(\ran\psi\cap\dom\varphi).
	\]
Note that the inverse of $\varphi$ is simply $\varphi^{-1}:\ran\varphi\to\dom\varphi$.

Following Exel~\cite{E1} by {\it an action of an inverse semigroup $S$ on $X$} we mean a homomorphism $S\to\cI X$ (in semigroup theory such homomorphisms are called representations of $S$, see~\cite[7.3]{Clifford-Preston-2}). If $S$ is a monoid, then the term ``action'' will always mean ``unital action'', i.\,e. a homomorphism of monoids $S\to\cI X$. It was proved in~\cite{E1} that there is a one-to-one correspondence between partial actions of a group $G$ on a set $X$ and (unital) actions of $\cS G$ on $X$, where $\cS G$ is the semigroup generated by $[x]$ ($x\in G$) subject to the following relations:
	\begin{enumerate}[(i)]
		\item $[x\m][x][y]=[x\m][xy]$; 
		\item $[x][y][y\m]=[xy][y\m]$;  
		\item $[x][1_G]=[x]$. 
	\end{enumerate}
It was shown in~\cite{E1} that   $\cS G$ is an inverse monoid with $1_{\cS G}=[1_G]$ and $[x]\m=[x\m]$. Moreover, any element of $\cS G$ can be expressed in the form $\e_{x_1}\dots\e_{x_n}[y]$, where $\e_{x_i}$ denotes $[x_i][x_i\m]\in E(\cS G)$, all $x_i$ are different, non-identity and not equal to $y$. This expression is unique up to a permutation of the idempotents $\e_{x_i}$. It follows that each idempotent of $\cS G$ has the form $\e_{x_1}\dots\e_{x_n}$ for some (uniquely defined) set $x_1,\dots,x_n\in G$. As it was mentioned in the introduction, $\cS G$ is isomorphic to the Birget-Rhodes expansion of $G$ \cite{KL}. 

Notice that, given a semigroup $A$, the composition of two partial isomorphisms of $A$ (i.\,e. isomorphisms between arbitrary  ideals of $A$)  is not a partial isomorphism in general, since the domain of the composition is not necessarily an ideal. However,  the set  $\cIui A$ of all isomorphisms between unital ideals of $A$ (i.\,e. ideals generated by central idempotents of $A$) forms an inverse semigroup. If $A$ is a monoid, then $\id_A\in\cIui A$, so $\cIui A$ is a monoid.

	\begin{defn}\label{defn-act_of_inv_semigr_on_semigr}
		Let $S$ be an inverse semigroup. An action of $S$ on a semigroup $A$ is a homomorphism of semigroups $\tau:S\to\cIui A$. If $S$ and $A$ are monoids, then $\tau$ is called unital, whenever it is a homomorphism of monoids.
	\end{defn}

By~\cite[Proposition 4.1]{E1} a partial action of $G$ on a set $X$ can be seen as a partial representation of $G$ in $ \cI X$.  This immediately implies that unital partial actions of $G$ on a monoid $A$ can be identified with partial representations of $G$ in $\cIui A$. Then~\cite[Proposition 2.2]{E1} yields the next.

	\begin{prop}\label{prop-Exel_th_for_act_on_semigr}
		There is a one-to-one correspondence between partial $G$-modules and unital actions of $\cS G$ on commutative monoids.
	\end{prop}

\noindent More precisely, let $(A,\0)\in\pMod G.$ Then the corresponding action $\tau $ of $\cS G$ on $A$  is given as follows. For arbitrary   $s= \e_{x_1}\dots\e_{x_n}[y] \in \cS G,$ $\tau_s$ is the bijection 
	\[
		1_{y\m x_1}\dots 1_{y\m x_n}1_{y\m } A \ni a  \mapsto  \0_y(a) \in  1_{x_1}\dots 1_{x_n}1_y A.
	\]   
This correspondence is in fact an isomorphism of categories as explained  below.

	\begin{rem}\label{rem-from_morph_p_mod_G_to_morph_act_S(G)}
		Let $A,A'\in\pMod G$, $\tau$ and $\tau'$ be the corresponding actions of $\cS G$ on the monoids $A$ and $A'$. Denote by $1_s$ and $1_s'$ the identity elements of $\ran{\tau_s}$ and $\ran{\tau_s'}$, respectively ($s\in\cS G$). Then a homomorphism of semigroups $\varphi:A\to A'$ is a morphism of partial $G$-modules if and only if
	\begin{enumerate}[(i)]
		\item $\varphi(1_s)=1_s'$, $s\in\cS G$ (so $\varphi(\dom{\tau_s})\subseteq\dom{\tau_s'}$, $s\in\cS G$);
		\item $\varphi\circ\tau_s=\tau_s'\circ\varphi$ on $\dom{\tau_s}$, $s\in\cS G$.
	\end{enumerate}
	\end{rem}

It is reasonable to give the next.

	\begin{defn}\label{defn-morph_of_act_of_inv_semi}
		Let $\tau$ and $\tau'$ be actions of an inverse semigroup $S$ on semigroups $A$ and $A'$, respectively. {\it A morphism}  $\tau \to \tau'$ is a homomorphism $\varphi:A\to A'$, such that the conditions (i) and (ii) from Remark~\ref{rem-from_morph_p_mod_G_to_morph_act_S(G)} hold (as above $\ran\tau_s=1_sA$ and $\ran\tau_s'=1_s'A'$, $s\in S$).
	\end{defn}

	The category of actions (unital actions) of an inverse semigroup (monoid) $S$ on commutative semigroups (monoids) will be denoted by $\A S$.

	\begin{prop}\label{prop-pMod(G)_is_isom_to_A(S(G))}
		The categories $\pMod G$ and $\A{\cS G}$ are isomorphic.
	\end{prop}
\noindent Indeed, Proposition~\ref{prop-Exel_th_for_act_on_semigr} and Remark~\ref{rem-from_morph_p_mod_G_to_morph_act_S(G)} give a functor $\pMod G\to\A{\cS G}$ which is bijective on objects and identity on morphisms.

Next we are going to give a characterization of actions of an inverse semigroup $S$ on a semigroup $A$ in terms of endomorphisms of $A$. 
	\begin{lem}\label{lem-im_of_ident_under_act_of_inv_semigr} 
		Let $\tau:S\to\cIui A$, $s\mapsto\tau_s$, be an action of an inverse semigroup $S$ on a semigroup $A$ and $1_s$ denote the identity of $\ran\tau_s$. Then
		\begin{enumerate}[(i)]
			\item $\tau_e=\id_{1_eA}$ for all $e\in E(S)$;
			\item $\tau_s(1_{s^{-1}}1_t)=1_{st}$ for all $s,t\in S$.
		\end{enumerate}
	\end{lem}
	\begin{proof}
		Equality (i) follows by observing that $\tau_e$ is a bijection of $\dom\tau_e=\ran\tau_e$, which coincides with its square. For (ii) note that 
	\[
	    1_{st}A=\ran\tau_{st}=\ran{(\tau_s\circ\tau_t)}=\tau_s(\dom\tau_s\cap\ran\tau_t)=\tau_s(1_{s\m} A\cap 1_tA)=\tau_s(1_{s\m}1_t A).
	\]
Therefore, $1_{st}=\tau_s(1_{s\m}1_t)$.
	\end{proof}

For any semigroup $A$ denote by $C(A)$ the center of $A$. 

	\begin{thrm}\label{thrm-from_part_iso_to_endo}
		Let $S$ be an inverse semigroup and $A$ a semigroup. There is a one-to-one correspondence between actions of $S$ on $A$ and pairs $(\lambda,\alpha)$, where $\lambda$ is a homomorphism $S\to\End A$, $s\mapsto\lambda_s$, and $\alpha$ is a homomorphism $E(S)\to E(C(A))$ such that
	\begin{enumerate}[(i)]
		\item $\lambda_e(a)=\alpha(e)a$ for all $e\in E(S)$ and $a\in A$;
		\item $\lambda_s(\alpha(e))=\alpha(ses^{-1})$ for all $s\in S$ and $e\in E(S)$.
	\end{enumerate}
	\end{thrm}
	\begin{proof}
		Let $\tau:S\to\cIui A$, $s\mapsto\tau_s$, be an action of $S$ on $A$. Obviously, the identity $1_s$ of $\ran\tau_s$ belongs to $E(C(A))$, so we may define $\alpha^\tau:E(S)\to E(C(A))$ by 
	\begin{equation}\label{eq-alpha^theta}
		\alpha^\tau(e)=1_e.
	\end{equation}
According to (i) and (ii) of Lemma~\ref{lem-im_of_ident_under_act_of_inv_semigr} we have $1_{ef}=\tau_e(1_e1_f)=1_e1_f$, hence $\alpha^\tau$ is a homo\-mor\-phism. Furthermore, for any $s\in S$ and $a\in A$ set 
	\begin{equation}\label{eq-from_part_iso_to_endo}
		\lambda^\tau_s(a)=\tau_s(1_{s^{-1}}a).
	\end{equation} 
The right-hand side of~\eqref{eq-from_part_iso_to_endo} makes sense, because $1_{s^{-1}}$ is the identity of $\ran\tau_{s^{-1}}=\dom\tau_s$ and $\dom\tau_s$ is an ideal. Since $1_{s^{-1}}$ is a central idempotent and $\tau_s$ is a homomorphism, we see that
	\[
		\lambda^\tau_s(ab)=\tau_s(1_{s^{-1}}ab)=\tau_s(1_{s^{-1}}a\cdot 1_{s^{-1}}b)=
\tau_s(1_{s^{-1}}a)\tau_s(1_{s^{-1}}b)=\lambda^\tau_s(a)\lambda^\tau_s(b),
	\]
so $\lambda^\tau_s\in\End A$. To check whether $\lambda^\tau:S\to\End A$, $s\mapsto\lambda^\tau_s$, is a homomorphism, use (ii) of Lemma~\ref{lem-im_of_ident_under_act_of_inv_semigr}:
	\[
		\lambda^\tau_s(\lambda^\tau_t(a))
		=\tau_s(1_{s^{-1}}\tau_t(1_{t^{-1}}a))
		=\tau_s(\tau_t(1_{t^{-1}}1_{t^{-1}s^{-1}})\tau_t(1_{t^{-1}}a))
		=\tau_s(\tau_t(1_{t^{-1}}1_{t^{-1}s^{-1}}a)),
	\]
which is $\tau_{st}(1_{t^{-1}}1_{t^{-1}s^{-1}}a)$, because $1_{t^{-1}}1_{t^{-1}s^{-1}}a\in\dom\tau_{st}$. It remains to note that
	\[
		\tau_{st}(1_{t^{-1}}1_{t^{-1}s^{-1}}a)
		=\tau_{st}(1_{t^{-1}}1_{t^{-1}s^{-1}})\tau_{st}(1_{t^{-1}s^{-1}}a)
		=1_s\tau_{st}(1_{t^{-1}s^{-1}}a)
		=\tau_{st}(1_{t^{-1}s^{-1}}a)
		=\lambda^\tau_{st}(a).
	\]
Here we again applied (ii) of Lemma~\ref{lem-im_of_ident_under_act_of_inv_semigr} and the fact that $\ran\tau_{st}\subseteq\ran\tau_s$.

Now we verify that the pair $(\lambda^\tau,\alpha^\tau)$ satisfies (i) and (ii). The property (i) obviously follows from (i) of Lemma~\ref{lem-im_of_ident_under_act_of_inv_semigr}. For (ii) observe that
	\[
		\lambda^\tau_s(1_e) =\tau_s(1_{s^{-1}}1_e) =\tau_s(1_{s^{-1}}\tau_e(1_{s^{-1}}1_e))=
\tau_s(1_{s^{-1}}1_{es^{-1}})=1_{ses^{-1}},
	\] 
using (i) and (ii) of Lemma~\ref{lem-im_of_ident_under_act_of_inv_semigr}.

Conversely, given a pair $(\lambda,\alpha)$ with (i) and (ii), define $1_s\in E(C(A))$ by 
	\begin{equation}\label{eq-1_x=alpha(xx^-1)}
		1_s=\alpha(ss^{-1}).
	\end{equation}
Note that in view of (ii)
	\[
		\lambda_s(1_{s^{-1}})=\lambda_s(\alpha(s^{-1}s))=\alpha(ss^{-1}ss^{-1})=\alpha(ss^{-1})=1_s
	\]
and hence $\lambda_s(1_{s^{-1}}A)\subseteq 1_sA$. Moreover, $\lambda_s|_{1_{s^{-1}}A}$ is a bijection of $1_{s^{-1}}A$ onto $1_sA$ with inverse $\lambda_{s^{-1}}|_{1_sA}$, because
	\[
		(\lambda_s\circ\lambda_{s^{-1}})|_{1_sA}=\lambda_{ss^{-1}}|_{1_sA}=\id_{1_sA},
	\] 
 by (i) and similarly $(\lambda_{s^{-1}}\circ\lambda_s)|_{1_{s^{-1}}A}=\id_{1_{s^{-1}}A}$. Hence, we may define $\tau^{(\lambda,\alpha)}:S\to\cIui A$ by $\tau^{(\lambda,\alpha)}_s=\lambda_s|_{1_{s^{-1}}A}$. To prove that $\tau^{(\lambda,\alpha)}$ is a homomorphism, first observe that
	\begin{align*}
		\dom{(\tau^{(\lambda,\alpha)}_s\circ\tau^{(\lambda,\alpha)}_t)}&=
\lambda_{t^{-1}}(1_{s^{-1}}A\cap 1_tA)=\lambda_{t^{-1}}(1_{s^{-1}}1_tA)\\
&=\lambda_{t^{-1}}(\alpha(s^{-1}s))1_{t^{-1}}A=
\alpha(t^{-1}s^{-1}st)\alpha(t^{-1}t)A,
	\end{align*}
which is $\alpha(t^{-1}s^{-1}st)A=1_{t^{-1}s^{-1}}A=\dom\tau^{(\lambda,\alpha)}_{st}$, because $\alpha$ is a homomorphism. Furthermore
	\[
		\lambda_t(1_{t^{-1}s^{-1}})=\lambda_t(\alpha(t^{-1}s^{-1}st))= \alpha(t(t\m s^{-1}st) t\m )= \alpha(s^{-1}st t\m )= 1_{s^{-1}}1_t.
	\]
Then since $1_{t^{-1}s^{-1}}=1_{t^{-1}s^{-1}}1_{t^{-1}},$ 
for any $a\in\dom\tau^{(\lambda,\alpha)}_{st}$ we have
	\begin{align*}
		(\tau^{(\lambda,\alpha)}_s \circ \tau^{(\lambda,\alpha)}_t)(a)&=
\tau^{(\lambda,\alpha)}_s(\tau^{(\lambda,\alpha)}_t(1_{t^{-1}s^{-1}}a))=
\tau^{(\lambda,\alpha)}_s(\tau^{(\lambda,\alpha)}_t(1_{t^{-1}s^{-1}}1_{t^{-1}}a))\\
&=\tau^{(\lambda,\alpha)}_s(\lambda_t(1_{t^{-1}s^{-1}}1_{t^{-1}}a))=
\tau^{(\lambda,\alpha)}_s(1_{s^{-1}}\lambda_t(1_{t^{-1}}a))\\
&=\lambda_s(1_{s^{-1}}\lambda_t(1_{t^{-1}}a))=1_s\lambda_{st}(1_{t^{-1}}a)\\
&=1_s\lambda_{st}(a)=1_s\tau^{(\lambda,\alpha)}_{st}(a)=\tau^{(\lambda,\alpha)}_{st}(a).
	\end{align*}

It remains to show that $\tau\mapsto(\lambda^\tau,\alpha^\tau)$ and $(\lambda,\alpha)\mapsto\tau^{(\lambda,\alpha)}$ are mutually inverse. Consider $\tau^{(\lambda^\tau,\alpha^\tau)}$. For any $s\in S$ the identity of $\dom\tau^{(\lambda^\tau,\alpha^\tau)}_s$ is $\alpha^\tau(s^{-1}s)=1_{s^{-1}s}$, which is the identity of $\dom\tau_{ss^{-1}}$. But using (ii) of Lemma~\ref{lem-im_of_ident_under_act_of_inv_semigr} we see that 
	\[
		1_{s^{-1}}=\tau_{s^{-1}}(1_s)=\tau_{s^{-1}}(1_s\cdot 1_s)=1_{s^{-1}s},
	\]
so $\dom\tau^{(\lambda^\tau,\alpha^\tau)}_s=\dom\tau_s$. By our construction 
	\[
		\tau^{(\lambda^\tau,\alpha^\tau)}_s=\lambda^\tau_s|_{1_{s^{-1}s}A}=
\lambda^\tau_s|_{1_{s^{-1}}A}.
	\]
But $\lambda^\tau_s(a)=\tau_s(1_{s^{-1}}a)$ and $a=1_{s^{-1}}a$ for $a\in 1_{s^{-1}}A$, hence $\lambda^\tau_s=\tau_s$ on $1_{s^{-1}}A ,$ and consequently, $\tau^{(\lambda^\tau,\alpha^\tau)}=\tau .$

Now given $(\lambda,\alpha)$, we  show that $\alpha^{\tau^{(\lambda,\alpha)}}=\alpha$ and $\lambda^{\tau^{(\lambda,\alpha)}}=\lambda$. For any $e\in E(S)$ the image $\alpha^{\tau^{(\lambda,\alpha)}}(e)$ is the identity of $\ran\tau^{(\lambda,\alpha)}_e$, which is $\alpha(ee^{-1})=\alpha(e)$. Furthermore, using (i) and (ii) we get
	\begin{align*}
		\lambda^{\tau^{(\lambda,\alpha)}}_s(a)&=\tau _s^{(\lambda,\alpha)}(\alpha(s^{-1}s)a)=
\lambda_s|_{\alpha(s^{-1}s)A}(\alpha(s^{-1}s)a)=\lambda_s(\alpha(s^{-1}s)a)
=\lambda_s(\alpha(s^{-1}s))\lambda_s(a)\\
&=\alpha(ss^{-1}ss^{-1})\lambda_s(a)=\alpha(ss^{-1})\lambda_s(a)
=\lambda_{ss^{-1}}(\lambda_s(a))=\lambda_{ss^{-1}s}(a)=\lambda_s(a).
	\end{align*}
	\end{proof}

	\begin{rem}\label{rem-unital-tau-in-terms-of-(lambda,alpha)}
	Under the conditions of Theorem~\ref{thrm-from_part_iso_to_endo} assume that $A$ is a monoid. Then the unital actions of $S$ on $A$ correspond to those pairs $(\lambda,\alpha)$, in which $\lambda$ and $\alpha$ are homomorphisms of monoids.
	\end{rem}

	\begin{defn}\label{defn-S-mod}
		Given an inverse semigroup $S$, {\it an $S$-module} is a triple  $(A,\lambda,\alpha)$ (often written shortly as $A$), where $A$ is a commutative semigroup, $\lambda:S\to\End A$ and $\alpha:E(S)\to E(A)$ are homomorphisms satisfying (i) and (ii) of Theorem~\ref{thrm-from_part_iso_to_endo}. If $\alpha$ is an isomorphism, then we say that $A$ is {\it strict}.
	\end{defn}

	If $S$ is a monoid, then any $S$-module is assumed to be unital in the sense that $A$ is a monoid and $\lambda$, $\alpha$ are homomorphisms of monoids.

%Notice that in~\cite{Lausch} the term $S$-module means our strict inverse  $S$-module.

	\begin{prop}\label{prop-from_morph_of_A(S(G))_to_morph_Mod(S(G))}
		Let $\tau$ and $\tau'$ be actions of an inverse semigroup $S$ on semigroups $A$ and $A'$, respectively. Set $(\lambda,\alpha)=(\lambda^\tau,\alpha^\tau)$ and $(\lambda',\alpha')=(\lambda^{\tau'},\alpha^{\tau'})$. Then a homomorphism $\varphi:A\to A'$ is a morphism  $\tau\to\tau'$ if and only if it satisfies
	\begin{enumerate}[(i)]
		\item $\varphi\circ\alpha=\alpha'$ on $E(S)$; 
		\item $\varphi\circ\lambda_s=\lambda_s'\circ\varphi$ for all $s\in S$.
	\end{enumerate}
	\end{prop}
	\begin{proof}
		Thanks to~\eqref{eq-1_x=alpha(xx^-1)} condition (i) is equivalent to $\varphi(1_s)=1_s'$, while (ii) is the same as $\varphi(\tau_s(1_{s\m}a))=\tau_s'(1_{s\m}'\varphi(a))$ for all $a\in A$ by~\eqref{eq-from_part_iso_to_endo}. In view of (i) the latter can be replaced by $\varphi(\tau_s(1_{s\m}a))=\tau_s'(\varphi(1_{s\m}a))$ for all $a\in A$, i.\,e. $\varphi\circ\tau_s=\tau_s'\circ\varphi$ on $1_{s\m}A=\dom{\tau_s}$.
	\end{proof}

	\begin{defn}\label{defn-morph_of_Mod(S)}
		Let $(A,\lambda,\alpha)$ and $(A',\lambda',\alpha')$ be modules over an inverse semigroup $S$. By {\it a morphism}  $(A,\lambda,\alpha) \to (A',\lambda',\alpha')$   we mean a homomorphism $\varphi:A\to A'$ satisfying (i) and (ii) of Proposition~\ref{prop-from_morph_of_A(S(G))_to_morph_Mod(S(G))}.
	\end{defn}

	Given an inverse semigroup (monoid) $S$, denote by $\Mod S$ the category of $S$-modules (unital $S$-modules) and their morphisms. Then Theorem~\ref{thrm-from_part_iso_to_endo} and Proposition~\ref{prop-from_morph_of_A(S(G))_to_morph_Mod(S(G))}   yield the next. 

	\begin{prop}\label{prop-A(S)_is_isom_to_Mod(S)}
		For any inverse semigroup $S$ the categories $\A S$ and $\Mod S$ are isomorphic.
	\end{prop}

	\begin{cor}\label{cor-pMod(G)_is_isom_to_Mod(S(G))}
		The category $\pMod G$ is isomorphic to $\Mod{\cS G}$.
	\end{cor}

This can be specified as follows. Given $(A,\0)\in\pMod G$, the corresponding $\lambda^\0$ is defined on a generator $[x]\in\cS G$ by $\lambda^\0_{[x]}(a)=\0_x(1_{x\m}a)$ and $\alpha^\0(\e_x)=1_x$ for all $x\in G$. This implies that
	\begin{equation}\label{eq-lambda_and_alpha_for_S(G)}
		\lambda^\0_{\e_{x_1}\dots\e_{x_n}[y]}(a)=1_{x_1}\dots 1_{x_n}\0_y(1_{y\m}a),\ \ \alpha^\0(\e_{x_1}\dots\e_{x_n})=1_{x_1}\dots 1_{x_n}.
	\end{equation}
For any morphism $\varphi:(A,\0)\to(A',\0')$ we see that 
	\[
		\varphi(\alpha(\e_{x_1}\dots\e_{x_n}))=\varphi(1_{x_1}\dots 1_{x_n})=1_{x_1}'\dots 1_{x_n}'=\alpha'(\e_{x_1}\dots\e_{x_n})
	\]
and
	\begin{align*}
		\varphi(\lambda_{\e_{x_1}\dots\e_{x_n}[y]}(a))&=\varphi(1_{x_1}\dots 1_{x_n}\0_y(1_{y\m}a))=1_{x_1}'\dots 1_{x_n}'\0'_y(\varphi(1_{y\m}a))\\
		&=1_{x_1}'\dots 1_{x_n}'\0'_y(1_{y\m}'\varphi(a))=\lambda_{\e_{x_1}\dots\e_{x_n}[y]}'(\varphi(a))
	\end{align*}
(here $(\lambda,\alpha)=(\lambda^\0,\alpha^\0)$ and $(\lambda',\alpha')=(\lambda^{\0'},\alpha^{\0'})$ as above).

In~\cite{Lausch} H. Lausch considered  (strict) $S$-module structures on semilattices of abelian groups $A$. The following remark shows that in the case of partial $G$-modules it is also reasonable to restrict our consideration to an inverse $A$.

	\begin{rem}\label{rem-semigr_of_U(prod_D_x)}
		Let $(A,\0)\in\pMod G$. Define $\tilde A$ to be the (inverse) subsemigroup of $A$ formed by the invertible elements of all ideals $1_{x_1}\dots 1_{x_n}A$, where $x_1,\dots,x_n\in G$, $n\in\bbN$. Then $\0_x(1_{x\m}\tilde A)=1_x\tilde A$, so $\theta$ restricted to $\tilde A$ defines a partial action $\tilde\0$ of $G$ on $\tilde A$. Moreover, $(\tilde A,\tilde\0)\in\pMod G$ and $H^n(G,A)=H^n(G,\tilde A)$.
	\end{rem}

	\noindent For if $a\in\cU{1_{x_1}\dots 1_{x_n}A}$ and $b\in\cU{1_{y_1}\dots 1_{y_m}A}$, then $ab\in\cU{1_{x_1}\dots 1_{x_n}1_{y_1}\dots 1_{y_m}A}$. The inverse of $a$ in $\tilde A$ is its inverse in the corresponding ideal. Furthermore, if $a\in\cU{1_{x_1}\dots 1_{x_n}A}$, then $\0_x(1_{x\m}a)\in\cU{1_x1_{xx_1}\dots 1_{xx_n} A }\subseteq 1_x\tilde A$ and since similarly $\0_{x\m}(1_x\tilde A)\subseteq 1_{x\m}\tilde A$, we have the equality $\0_x(1_{x\m}\tilde A)=1_x\tilde A$. This implies that
	\[
		\0_x(1_{x\m}1_y\tilde A)=\0_x(1_{x\m}1_y)\0_x(1_{x\m}\tilde A)=1_x1_{xy}\tilde A,
	\]
so the restrictions $\tilde\0_x$ of $\0_x$ to $1_{x\m}\tilde A\subseteq 1_{x\m}A$ are in fact a partial action of $G$ on $\tilde A$. By construction $(\tilde A,\tilde\0)\in\pMod G$. Note that $\cU{1_{x_1}\dots 1_{x_n}A}=\cU{1_{x_1}\dots 1_{x_n}\tilde A}$ and hence $C^n(G,A)=C^n(G,\tilde A)$. Since $\tilde\0$ is the restriction of $\0$, the coboundary homomorphisms of these two cochain complexes also coincide and thus, their cohomology groups are equal.

Notice that  
	\[
		\tilde{A}= \bigcup\cU{1_{x_1}\dots 1_{x_n}A},
	\] 
where $x_1, x_2, \dots, x_n$ run over $G.$ Observe that if $e,f\in A$ are  idempotents with $e \neq f$ then $\cU{eA} \cap \cU{fA}  = \emptyset .$ Hence  the above union can be made to be disjoint by removing  repetitions, and then it will give a representation of $\tilde A$ as a semilattice of abelian groups. In particular, any idempotent of $\tilde A$, being an element of some group component, is identity of this component, i.\,e. has a form $1_{x_1}\dots 1_{x_n}$ for $x_1,\dots,x_n\in G$.

	\begin{defn}\label{defn-inverse_part_mod}
		A partial $G$-module $(A,\0)$ is called {\it inverse} if $A$ is inverse and $E(A)$ is generated by the idempotents $1_x$ ($x\in G$).
	\end{defn}

	\begin{rem}\label{rem-tilde_A_is_inverse_part_mod}
		If $(A,\0)\in\pMod G$, then $(\tilde A,\tilde\0)$ is inverse.
	\end{rem}

	Another immediate observation is that $E(A)=\langle 1_x\mid x\in G\rangle$ if and only if $\alpha^\0:E(\cS G)\to E(A)$ is an epimorphism.

	\begin{defn}\label{defn-inv_S-mod}
		An $S$-module $(A,\lambda,\alpha)$ is called {\it inverse} if $A$ is inverse and $\alpha$ is an epimorphism.
	\end{defn}

	\begin{rem}\label{rem-inv_part_G-mod_is_inv_S(G)-mod}
		A partial $G$-module $(A,\0)$ is inverse if and only if the corresponding $\cS G$-module $(A,\lambda^\0,\alpha^\0)$ is inverse.
	\end{rem}

Thus, passing from an inverse partial $G$-module to an $\cS G$-module, we see that the only reason why we might not obtain an $\cS G$-module in the sense of  Lausch is that the map $\alpha^\0:E(\cS G)\to E(A)$ might not be injective in general. For example, considering a classical (i.\,e. usual) $G$-module $A$ (which is readily seen to be  a strict inverse $G$-module) we have $|E(A)|=1$, while $|E(\cS G)|=2^{|G|-1}$. Ne\-ver\-the\-less, an inverse partial $G$-module $A$ may be seen as a strict inverse module over some epimorphic image of $\cS G$. This epimorphic image can be chosen to be $E(A)*_\0 G$ as explained in what follows (see Corollary~\ref{cor-part_mod_as_epi-strict_S(G)-mod}). In particular, in the above example this construction leads to the initial $G$-module $A$.

Given an inverse semigroup $S$ and an $S$-module $(A,\lambda,\alpha)$ denote by $L(A,S)$ the set of $a\delta_s$, where $s\in S$, $a\in\alpha(ss\m)A$ and $\delta_s$ is a symbol.

	\begin{prop}\label{prop-L_is_a_semigroup}
		The set $L(A,S)$ is a semigroup under the multiplication $a\delta_s\cdot b\delta_t=a\lambda_s(b)\delta_{st}$.
	\end{prop}
	\begin{proof} It is directly  verified that  $L(A,S)$ can be seen as the $\lambda $-semidirect product defined in~\cite[p. 148]{Lawson}, so that the above multiplication is well-defined and associative. 
	\end{proof}

	\begin{rem}\label{rem-elements-in-L} 
		Notice that $\lambda_s(a)\delta_{s}$ and  $\lambda_s(a)\delta_{ss\m }$ are   elements in $L(A,S)$ for all $a\in A$, as $\lambda_s(a) =\lambda_{ss\m s}(a) = \lambda_{s s\m} (\lambda_s(a)) = \alpha (ss\m ) \lambda_s(a).$
	\end{rem}

It is well-known that any inverse semigroup $S$ is partially ordered in a following natural way: $s\le t$ if and only if there is $e\in E(S)$ such that $s=et$ (and in this case one can take $e=ss^{-1}$). This is the same as to say that $s=tf$ for some $f\in E(S)$. Notice that for $e,f\in E(S)$ we have $e\le f\Leftrightarrow ef=e$. 

	\begin{rem}\label{rem-idemp-of-L(A,S)} 
		$E(L(A,S))=\{f\delta_e\mid e\in E(S),\ \ f\in E(A),\ \ f\le \alpha(e)\}$.
	\end{rem}

Inspired by~\cite{EV}, instead of $L(A,S)$ we shall deal with a quotient of    $L(A,S)$ which takes into account the natural partial order on $S.$  This idea has important advantages, one of them being the fact that in the $C^{\ast}$-algebraic context the analogous construction fits the definition of the crossed product by an inverse semigroup given by N.~Sieben in~\cite{Sieben} (see~\cite[Theorem 5.6]{EV}). 

Denote by $\sigma$ the congruence generated by $\le$. Since $\le$ is compatible with multiplication in $S$ (\cite[Lemma 7.2]{Clifford-Preston-2}), $\sigma$ is in fact the minimal equivalence relation containing the partial order $\le$. 

	\begin{lem}\label{lem-sigma_is_min_gr_congr}
		For any inverse semigroup $S$ and for all $s,t\in S$ we have $(s,t)\in\sigma$ if and only if there is $u\in S$ such that $u\le s,t$.
	\end{lem}
	\begin{proof}
		We shall  prove the ``only if'' part (the ``if'' part is obvious). As it was mentioned before the lemma, $\sigma$ is the symmetric transitive closure of $\le$. Thus, $(s,t)\in\sigma$ means that there are $s=s_1,\dots,s_n=t\in S$ such that for all $i=1,\dots,n-1$ we have $s_i\le s_{i+1}$ or $s_{i+1}\le s_i$. Therefore, $s_i=e_is_{i+1}$ or $s_{i+1}=e_is_i$ for some $e_i\in E(S)$ ($i=1,\dots,n-1$). In any case $e_is_i=e_is_{i+1}$ for all $i=1,\dots,n-1$. Set $e=e_1\dots e_{n-1}$. Then
	\begin{align*}
		es=es_1&=(e_2\dots e_{n-1})(e_1s_1)=(e_2\dots e_{n-1})(e_1s_2)=(e_1e_3\dots e_{n-1})(e_2s_2)\\
	&=\dots=(e_1\dots e_{n-2})(e_{n-1}s_{n-1})=(e_1\dots e_{n-2})(e_{n-1}s_n)=es_n=et.
	\end{align*}
So, we can take $u=es=et\le s,t$.
	\end{proof}
\noindent This shows that   $\sigma$ is the minimum group congruence on $S$ (see~\cite[Theorem 2.4.1]{Lawson}). For $S=\cS G$ we have $(\e_{x_1}\dots\e_{x_n}[y],\e_{z_1}\dots\e_{z_m}[w])\in\sigma$ if and only if $y=w$ (see~\cite[Example 1.5]{EV}), so the group $\cS G/\sigma$ is isomorphic to $G$.

	\begin{defn}\label{defn-cross_prod_for_S-mod}
		For arbitrary $(A,\lambda,\alpha)\in\Mod S$ define $\rho$ to be the congruence on $L(A,S)$ ge\-ne\-rated by 
	\begin{equation}\label{eq-rel_on_L_gen_rho}
		\{(a\delta_s,a\delta_t)\in L(A,S)\times L(A,S)\mid s\le t\}.
	\end{equation}
The quotient semigroup $L(A,S)/\rho$ will be denoted by $A\rtimes S$ and called {\it the semidirect product of $A$ and $S$}.
	\end{defn} 

\noindent Notice that $s\leq t$ implies $ ss\m \leq t t\m$ and hence $\alpha (s s\m)A =  \alpha (s s\m t t\m)A \subseteq  \alpha (t t\m)A.$

	\begin{rem}\label{rem-cross_prod_is_EV-prod}
		Under the conditions of Definition~\ref{defn-cross_prod_for_S-mod}, $A\rtimes S$ is $A\rtimes_{\tau^{(\lambda,\alpha)}}^a S$ in the sense of Definition~3.5 from~\cite{EV} modified for actions of inverse semigroups on semigroups.
	\end{rem}

	\begin{rem}\label{rem-rho_is_trans_closure}
		Under the conditions of Definition~\ref{defn-cross_prod_for_S-mod} the relation~\eqref{eq-rel_on_L_gen_rho} is compatible with the multiplication in $L(A,S)$, so the congruence $\rho$ is the equivalence relation generated by~\eqref{eq-rel_on_L_gen_rho}.
	\end{rem}
	\noindent Indeed,  given $a\delta_s,a\delta_t\in L(A,S)$, such that $s\le t$, and an arbitrary $b\delta_u\in L(A,S)$, we see that
	\[
		(b\delta_u\cdot a\delta_s,b\delta_u\cdot a\delta_t)=(b\lambda_u(a)\delta_{us},b\lambda_u(a)\delta_{ut})\in\rho,
	\]
because $us\le ut$. Similarly,
	\[
		(a\delta_s\cdot b\delta_u,a\delta_t\cdot b\delta_u)=(a\lambda_s(b)\delta_{su},a\lambda_t(b)\delta_{tu})\in\rho,
	\]
if we show that $a\lambda_s(b)=a\lambda_t(b)$. Since $s=ss\m t$, then $a\lambda_s(b)=a\lambda_{ss\m t}(b)=a\alpha(ss\m)\lambda_t(b)=a\lambda_t(b)$, because $a=\alpha(ss\m)a$.

As an immediate consequence, we notice that $(a\delta_s,b\delta_t)\in\rho$ implies $a=b$ and $(s,t)\in\sigma$. The converse, however, is not true in general.

	\begin{lem}\label{lem-explicit_form_of_rho}
		For any $(A,\lambda,\alpha)\in\Mod S$ and for arbitrary $a\delta_s,b\delta_t\in L(A,S)$ we have $(a\delta_s,b\delta_t)\in\rho$ if and only if there is $u\in S$ such that $u\le s,t$ and $a=b\in\alpha(uu\m)A$.
	\end{lem}
	\begin{proof}
		The ``if'' part is easy: if $u\le s,t$ and $a=b\in\alpha(uu\m)A$, then $(a\delta_u,a\delta_s)$ and $(a\delta_u,a\delta_t)$ belong to the relation defined by~\eqref{eq-rel_on_L_gen_rho}, so by symmetry and transitivity of $\rho$ we have $(a\delta_s,a\delta_t)\in\rho$.

		Now suppose $(a\delta_s,b\delta_t)\in\rho$, then, as it was mentioned above, $a=b$. Let $a\delta_s=a\delta_{s_1},\dots,a\delta_{s_n}=a\delta_t$ be a sequence of elements of $L(A,S)$ such that $s_i=e_is_{i+1}$ or $e_is_i=s_{i+1}$ for all $i=1,\dots,n-1$ ($e_i\in E(S)$).  Since  $a\in\alpha(s_is_i\m)A$ for all $i=1,\dots, n ,$ then $a\in\alpha(e_i)A$, $i=1,\dots,n-1 .$ Therefore, $a\in\alpha(e)A$, where $e=e_1\dots e_{n-1}$. Now   setting $u=es$ as in the proof of Lemma~\ref{lem-sigma_is_min_gr_congr},  we see that $u\le s,t$ and, moreover, $a\in\alpha(e)\alpha(ss\m)A=\alpha(uu\m)A$.
	\end{proof}

	\begin{exm}\label{exm-Z_2_cup_0}
		Take the cyclic group $\bbZ_2=\langle a\mid a^2=1\rangle$ and set $S=\bbZ_2\cup\{0\}$. Then $S$  is a commutative inverse monoid, whose group components are $\bbZ_2$ and $\{0\}$. We can define a structure of a (strict) $S$-module on $S$ as follows: $\lambda_s(t)=ss^{-1}t$ and $\alpha(e)=e$ for all $s,t\in S$ and $e\in E(S)$. Then $(a,1)\in\sigma$, but $(a\delta_a,a\delta_1)\not\in\rho$.
	\end{exm}
	\noindent Clearly, $0\le a,1$, but $a$ and $1$ are incomparable, so the only $u\in S$ with $u\le a,1$ is $0$. Although $a\in\alpha(aa^{-1})S=\alpha(1)\cdot S=S$, we cannot find $u\le a,1$ such that $a\in\alpha(uu^{-1})S$, since otherwise $a=0$. By Lemma~\ref{lem-explicit_form_of_rho} $(a\delta_a,a\delta_1)\not\in\rho$.

	Recall that an inverse semigroup $S$ is called {\it $E$-unitary} if $es\in E(S)$ implies $s\in E(S)$ for all $s\in S$ and $e\in E(S)$. Equivalently, $\sigma $ coincides with the compatibility relation $\sim $ on $S$ (see~\cite[pp. 24, 66]{Lawson}).  For instance, $\cS G$ is $E$-unitary for any group $G$, while the semigroup $S$ from Example~\ref{exm-Z_2_cup_0} is not $E$-unitary, because $0a=0\in E(S)$ and $a\notin  E(S)$.

	\begin{rem}\label{rem-rho_for_E-unitary_semigr}
		If under the conditions of Lemma~\ref{lem-explicit_form_of_rho} the semigroup $S$ is $E$-unitary then $(a\delta_s,b\delta_t)\in\rho$ exactly when $(s,t)\in\sigma$ and $a=b$.
	\end{rem}
	\noindent We need only explain the ``if'' part. Suppose $(s,t)\in\sigma$  and $a=b\in\alpha(ss\m tt\m)A.$ Since $S$ is $E$-unitary,  $s \sim t$ and hence $t\m s\in E(S)$. Set $v=(tt\m)s=t(t\m s)\le s,t$. Note that $vv\m=ss\m tt\m$ and thus $a=b\in\alpha(vv\m)A$.

	\begin{cor}\label{cor-mult_a_delta_s_by_b}
		If $(a\delta_s,b\delta_t)\in\rho$, then for any $c\in A$ we have $(ca\delta_s,cb\delta_t)\in\rho$.
	\end{cor}
	\noindent Indeed, if $a=b\in\alpha(uu\m)A$, then $ca=cb\in\alpha(uu\m)A$.

	\begin{rem}\label{rem-rho_idemp_sep}
		Let $(A,\lambda,\alpha)\in\Mod S$ and $a\delta_e,b\delta_f\in L(A,S)$ with $e,f\in E(S)$. Then
$(a\delta_e,b\delta_f)\in\rho$ exactly when $a=b$.
	\end{rem}
	\noindent For $ef\le e,f$, and if $a=b$ then $a=b\in\alpha(e)\alpha(f)A=\alpha(ef)A$.

Taking into account Remark~\ref{rem-idemp-of-L(A,S)}, we see that $\rho$ is idempotent-separating, which means that each $\rho$-class contains at most one idempotent of $L(A,S)$. Equivalently, $\rho^\natural$ restricted to $E(L(A,S))$ is injective, where $\rho^\natural$ denotes the natural epimorphism $L(A,S)\to A\rtimes S$.

For reader's convenience we also notice the following straightforward fact: if $\varphi : S_1 \to S_2$ is an epimorphism of inverse semigroups, then $E(S_2)=\varphi (E(S_1)).$ Indeed, taking $e\in E(S_2)$ and $s\in S_1$ with $\varphi (s)=e$ we have $\varphi (s s\m)=e,$ $s s\m \in E(S_1).$    

	\begin{prop}\label{prop-from_part_act_to_Lausch}
		Let $(A,\lambda,\alpha)$ be an $S$-module, such that $\alpha$ is an epimorphism. Then there exist an inverse semigroup $S'$, an epimorphism $\pi:S\to S'$ and an $S'$-module $(A,\tilde\lambda,\tilde\alpha)$, such that
	\begin{enumerate}[(i)]
		\item $(A,\tilde\lambda,\tilde\alpha)$ is strict;
		\item $\tilde\lambda\circ\pi=\lambda$ on $S$ and $\tilde\alpha\circ\pi=\alpha$ on $E(S)$.
	\end{enumerate}
Conversely, given an epimorphism $\pi:S\to S'$ and a strict $S'$-module $(A,\tilde\lambda,\tilde\alpha)$, the maps $\lambda$ and $\alpha$, determined by  (ii), endow  $A$ with  an $S$-module structure, and moreover $\alpha$ is surjective.
	\end{prop}
	\begin{proof}
		Note that the map $\varphi:S\to L(A,S)$ defined by
		\[
			\varphi(s)=\alpha(ss\m)\delta_s
		\]
is a homomorphism, because
	\[
		\varphi(s)\varphi(t)=\alpha(ss\m)\lambda_s(\alpha(tt\m))\delta_{st}=\alpha(stt\m s\m)\delta_{st}=\varphi(st).
	\]
Hence, $\im\varphi$ is an inverse subsemigroup of $L(A,S)$ with $(\alpha(ss\m)\delta_s)\m=\alpha(s\m s)\delta_{s\m}$. 

	Set $\pi=\rho^\natural\circ\varphi:S\to A\rtimes S$. It is an epimorphism onto $S'=\im\pi$. Clearly, $S'$ is also inverse and 
	\begin{equation}\label{eq-inv_of_pi(s)}
		\pi(s)\m=\pi(s\m)=\rho^\natural(\alpha(s\m s)\delta_{s\m}).
	\end{equation}
Note that for any $a\in A$ and $s\in S$ we have
	\[
		\varphi(s)\cdot a\alpha(s\m s)\delta_{s\m}=\alpha(ss\m)\delta_s\cdot a\alpha(s\m s)\delta_{s\m}=\alpha(ss\m)\lambda_s(a\alpha(s\m s))\delta_{ss\m}
	=\lambda_s(a)\delta_{ss\m},
	\] 
using Remark~\ref{rem-elements-in-L}. Therefore, applying $\rho^\natural$ we get
	\begin{equation}\label{eq-act_of_pi(s)}
		\pi(s)\rho^\natural(a\alpha(s\m s)\delta_{s\m})=\rho^\natural(\lambda_s(a)\delta_{ss\m})
	\end{equation}
for arbitrary $s\in S$ and $a\in A$. If $\pi(s)=\pi(t)$, then by~\eqref{eq-inv_of_pi(s)} and Corollary~\ref{cor-mult_a_delta_s_by_b} we have $\rho^\natural(a\alpha(s\m s)\delta_{s\m})=\rho^\natural(a\alpha(t\m t)\delta_{t\m})$ for arbitrary $a\in A$. So, the left-hand sides of~\eqref{eq-act_of_pi(s)} corresponding to $s$ and $t$ are equal, hence the right-hand sides are also equal. Then by the ``only if'' part of Lemma~\ref{lem-explicit_form_of_rho}, $\lambda_s=\lambda_t$. Thus, there exists a map $\tilde\lambda:S'\to\End A$ with $\tilde\lambda\circ\pi=\lambda$. It is readily seen that $\tilde\lambda$ is a homomorphism. In view of the epimorphism $\pi : S \to S'$  the idempotents of $S'$ are precisely the classes $\pi(e)=\rho^\natural(\alpha(e)\delta_e)$, where $e\in E(S)$, and $\pi(e)=\pi(f)$ if and only if $\alpha(e)=\alpha(f)$, because $\rho$ is idempotent-separating. This defines a monomorphism $\tilde\alpha:E(S')\to E(A)$ satisfying $\tilde\alpha\circ\pi=\alpha$ on $E(S)$. It is in fact an isomorphism 
due to the surjectivity of $\alpha$. We only need to check the properties (i) and (ii) of  Theorem~\ref{thrm-from_part_iso_to_endo}.

In view of the surjectivity of $\pi$, it  suffices  to prove
	\begin{enumerate}[(i')]
		\item $\tilde\lambda_{\pi(e)}(a)=\tilde\alpha(\pi(e))a$ for all $e\in E(S)$ and $a\in A$,
		\item $\tilde\lambda_{\pi(s)}(\tilde\alpha(\pi(e)))
=\tilde\alpha(\pi(s)\pi(e)\pi(s^{-1}))$ for all $s\in S$ and $e\in E(S)$.
	\end{enumerate}
But these are precisely the properties (i) and (ii) of  Theorem~\ref{thrm-from_part_iso_to_endo} for $(\lambda,\alpha)$ by the definition of $(\tilde\lambda,\tilde\alpha)$.

The converse is straightforward: if $\pi:S\to S'$ is an epimorphism, $\tilde\lambda:S'\to\End A$ is a homomorphism and $\tilde\alpha:E(S')\to E(A)$ is an isomorphism such that $(A, \tilde\lambda, \tilde\alpha)$ is a strict $S'$-module, then $\lambda=\tilde\lambda\circ\pi:S\to\End A$ and $\alpha=\tilde\alpha\circ\pi:E(S)\to E(A)$ are  homomorphisms. Moreover, $\alpha$ is surjective, because $\pi(E(S))=E(S')$. As shown above, (i) and (ii) of  Theorem~\ref{thrm-from_part_iso_to_endo} for $(\lambda,\alpha)$   are equivalent to the same properties  for  $(\tilde\lambda,\tilde\alpha)$. So, $(A,\lambda,\alpha)$ is an $S$-module with surjective $\alpha$. 
	\end{proof}

	\begin{defn}\label{defn-A*G}
		Given $(A,\0)\in\pMod G$, define {\it the crossed product of $A$ and $G$} to be the set $A*_\0 G$ of $a\delta_x$, where $x\in G$, $a\in 1_xA$ and $\delta_x$ is a symbol. It is a semigroup under the multiplication $a\delta_x\cdot b\delta_y=a\0_x(1_{x\m}b)\delta_{xy}$.
	\end{defn}
	\noindent The associativity of multiplication follows from the fact that all the domains of $\0$ are unital ideals of $A$ (see~\cite[Corollary~3.2]{DE}).

	\begin{rem}\label{rem-A*G_is_isom_to_A*S(G)}
		Let $(A,\0)\in\pMod G$ and $(A,\lambda,\alpha)\in\Mod{\cS G}$ the corresponding $\cS G$-module. Then $A\rtimes\cS G$ is isomorphic to $A*_\0 G$.
	\end{rem}
	\noindent On one hand, this follows from Remark~\ref{rem-cross_prod_is_EV-prod} and~\cite[Theorem~3.7]{EV} and, on the other hand, taking into account Remark~\ref{rem-rho_for_E-unitary_semigr}, an easy direct verification shows that the map  given by $A\rtimes\cS G\ni\rho^\natural(a\delta_{[y]})\leftrightarrow a\delta_y\in A*_\0 G$ is an isomorphism.

	\begin{cor}\label{cor-part_mod_as_epi-strict_S(G)-mod}
		Let $(A,\0)$ be an inverse partial $G$-module. Then there are an epimorphism $\pi:\cS G\to S=E(A)*_\0 G$ defined by $\pi(\e_{x_1}\dots\e_{x_n}[y])=1_{x_1}\dots 1_{x_n}1_y\delta_y$ and a strict inverse $S$-module $(A,\lambda,\alpha)$, such that
	\begin{align}\label{eq-lambda_and_alpha-for-S}
	    \alpha(\pi(\e_x))=1_x,\ \ \lambda_{\pi([x])}(a)=\0_x(1_{x\m}a)
	\end{align}
	for $x\in G$ and $a\in A$. Conversely, any pair consisting of an epimorphism $\pi:\cS G\to S$ and a strict inverse $S$-module $(A,\lambda,\alpha)$ induces by~\eqref{eq-lambda_and_alpha-for-S} a structure of an inverse partial $G$-module $(A,\0)$ on $A$.
	\end{cor}
	\noindent Indeed, let $(A,\lambda^\0,\alpha^\0)$ be the (inverse) $\cS G$-module corresponding to $(A,\0)$. Then there is $\pi:\cS G\to A\rtimes\cS G$ as in Proposition~\ref{prop-from_part_act_to_Lausch}, namely
	\[
		\pi(\e_{x_1}\dots\e_{x_n}[y])=
\rho^\natural(\alpha^\0(\e_{x_1}\dots\e_{x_n}\e_y)\delta_{\e_{x_1}\dots\e_{x_n}[y]})=
\rho^\natural(1_{x_1}\dots 1_{x_n}1_y\delta_{\e_{x_1}\dots\e_{x_n}[y]}),
	\]
which can be identified with $1_{x_1}\dots 1_{x_n}1_y\delta_y\in E(A)*_\0 G$ by Remark~\ref{rem-A*G_is_isom_to_A*S(G)}. Obviously, each element of $E(A)*_\0 G$ has this form. It remains to note that the strict inverse module $(A,\lambda,\alpha)$ over $\pi(\cS G)$ coming from $(A,\lambda^\0,\alpha^\0)$ satisfies
    \[
        \alpha(\pi(\e_x))=\alpha^\0(\e_x)=1_x,\ \ \lambda_{\pi([x])}(a)=\lambda^\0_{[x]}(a)=\0_x(1_{x\m}a)
    \]
    by (ii) of Proposition~\ref{prop-from_part_act_to_Lausch} and~\eqref{eq-lambda_and_alpha_for_S(G)}. For the converse, having a pair $\pi:\cS G\to S$, $(A,\lambda,\alpha)$ with the above properties, construct the inverse $\cS G$-module $(A,\lambda\circ\pi,\alpha\circ\pi)$ using Proposition~\ref{prop-from_part_act_to_Lausch}. Then $(A,\0)$ from~\eqref{eq-lambda_and_alpha-for-S} is the corresponding partial $G$-module by~\eqref{eq-lambda_and_alpha_for_S(G)}. It is inverse thanks to Remark~\ref{rem-inv_part_G-mod_is_inv_S(G)-mod}.

	\begin{rem}\label{rem-phi-hat} 
		Let $\varphi:A'\to A''$ be a morphism of $S$-modules. Then there is a (unique) homomorphism of semigroups 
$\hat{\varphi}: A'\rtimes S \to A''\rtimes S$ such that $\hat{\varphi}({\rho'}^\natural(a\delta_s)) = 
{\rho''}^\natural(\varphi(a)\delta_s)$ for any $a\delta_s\in L(A',S).$
	\end{rem}
	\noindent In order to see this, notice that by Definition~\ref{defn-morph_of_Mod(S)} the map $\bar{\varphi}:L(A',S)\ni a\delta_s\mapsto\varphi(a)\delta_s\in L(A'',S)$ is a well-defined homomorphism of semigroups, and in view of Lemma~\ref{lem-explicit_form_of_rho} $\bar\varphi(\rho')\subseteq\rho''$, so that one can define a homomorphism $\hat\varphi:A'\rtimes S \to A''\rtimes S$ by sending  ${\rho'}^\natural(a\delta_s)$ to 
${\rho''}^\natural(\bar{\varphi}(a\delta_s))={\rho''}^\natural(\varphi(a)\delta_s)$. 

	\begin{rem}\label{rem-morph_of_part_G-mod_in_terms_of_epimorph_im}
		Let $(A',\lambda',\alpha')$ and $(A'',\lambda'',\alpha'')$ be $S$-modules, $(A',\tilde{\lambda'},\tilde{\alpha'})$ and $(A'',\tilde{\lambda''},\tilde{\alpha''})$ the corresponding strict modules over the epimorphic images $S'$ and $S''$ of $S$ under $\pi':S\to S'$ and $\pi'':S\to S''$ as given in Proposition~\ref{prop-from_part_act_to_Lausch}. Then a homomorphism $\varphi:A'\to A''$ is a morphism of $S$-modules if and only if
	\begin{enumerate}[(i)]
		\item $\varphi\circ\tilde{\alpha}'\circ\pi'=\tilde{\alpha}''\circ\pi''$ on $E(S)$;
		\item $\varphi\circ \tilde{\lambda}'_{\pi'(s)}=\tilde{\lambda}''_{\pi''(s)}\circ\varphi$ for all $s\in S$.
	\end{enumerate}
Moreover, there is a (unique) homomorphism $\psi:S'\to S''$ satisfying $\psi\circ\pi'=\pi''$, so that (i) and (ii) can be replaced by
	\begin{enumerate}[(i')]
		\item $\varphi\circ\tilde{\alpha}'=\tilde{\alpha}''\circ\psi$ on $E(S')$;
		\item $\varphi\circ\tilde{\lambda}'_s=\tilde{\lambda}''_{\psi(s)}\circ\varphi$  for all $s\in S'$.
	\end{enumerate}
	\end{rem}
	\noindent Indeed, by (ii) of Proposition~\ref{prop-from_part_act_to_Lausch} the equalities (i) and (ii) are precisely (i) and (ii) from Proposition~\ref{prop-from_morph_of_A(S(G))_to_morph_Mod(S(G))} written for $\varphi:A'\to A''$. Furthermore, thanks to the fact that $\varphi\circ \alpha'=\alpha'',$ the homomorphism  $\hat\varphi: A'\rtimes S \to A''\rtimes S$ from Remark~\ref{rem-phi-hat} restricts to a homomorphism $\psi:S'\to S''$ with $\psi \circ {\pi}' = {\pi}''.$  Finally, replacing $\pi''$ with $\psi\circ\pi'$ in (i) and (ii) by surjectivity of $\pi'$ we get (i') and (ii').

	\begin{defn}\label{defn-pi_strict_S-mod}
		Given an inverse semigroup $S$ and an epimorphism $\pi:S\to S'$, we define the concept of a {\it $\pi$-strict} $S$-module as a pair $(A,\pi)$, where $A=(A, \lambda, \alpha)$ is a strict module over $S'=\pi(S)$. 
	\end{defn}

	Note that $\id$-strict $S$-modules are precisely strict $S$-modules in the sense of Definition~\ref{defn-S-mod}. Sometimes we shall omit $\pi$ in $(A,\pi)$ and call $A$ a $\pi$-strict $S$-module. Moreover, if $\pi$ is not specified, we say that $A$ is an {\it epi-strict} $S$-module.

	\begin{defn}\label{defn-morph_of_pi_strict_S-mod}
		Let $(A',\lambda',\alpha')$ and $(A'',\lambda'',\alpha'')$ be epi-strict $S$-modules under some epimorphisms $\pi':S\to S'$ and $\pi'':S\to S''$, respectively. By {\it a morphism of epi-strict $S$-modules} $(A',\pi')\to(A'',\pi'')$ we mean a pair $(\varphi,\psi)$, where $\varphi:A'\to A''$ and $\psi:S'\to S''$ are homomorphisms of semigroups such that
	\begin{enumerate}[(i)]
		\item $\psi\circ\pi'=\pi''$;
		\item $\varphi\circ\alpha'=\alpha''\circ\psi$ on $E(S')$;
		\item $\varphi\circ\lambda_s'=\lambda_{\psi(s)}''\circ\varphi$ for all $s\in S'$.
	\end{enumerate}
	\end{defn}

	It is immediately seen that in the case $S'=S''$, $\pi'=\pi''$ we have $\psi=\id$ and the equalities (ii)--(iii) give the definition of a morphism of strict $S'$-modules.

		Epi-strict $S$-modules and their morphisms form a category (under the coordinatewise composition of morphisms) which will be denoted by $\ESMod S$.

	\begin{defn}\label{defn-stand_epi-str_mod}
		An epi-strict $S$-module isomorphic to $(A,\tilde\lambda,\tilde\alpha)$ given in Proposition~\ref{prop-from_part_act_to_Lausch} will be called {\it standard}. 
	\end{defn}

We are going to describe the standard epi-strict $S$-modules $(A,\pi)$ in terms of $\pi$. Note that for $(A,\pi)$ from Proposition~\ref{prop-from_part_act_to_Lausch} we have $\ker\pi\subseteq\sigma$ (recall that for a homomorphism of semigroups $\phi:S_1\to S_2$ its kernel is defined to be the congruence $\ker\phi=\{(a,b)\in S_1\mid\phi(a)=\phi(b)\}$). It turns out that this condition characterizes $(A,\pi)$ as a standard epi-strict module, if $S$ is $E$-unitary.

	\begin{prop}\label{prop-desc_of_ESMod^*(S)}
		Let $S$ be an $E$-unitary inverse semigroup. An epi-strict $S$-module $(A,\lambda',\alpha')$, $\pi':S\to S'$ is standard if and only if $\ker{\pi'}\subseteq\sigma$.
	\end{prop}
	\begin{proof}
		The ``only if'' part is obvious: if $(\varphi,\psi):(A',\pi')\to(A'',\pi'')$ is an isomorphism between some  epi-strict $S$-modules, then $\psi\circ\pi'=\pi''$ implies $\ker{\pi'}=\ker{\pi''}$.

	For the ``if'' part consider the $S$-module $(A,\lambda,\alpha)$, where $\lambda=\lambda'\circ\pi'$ and $\alpha=\alpha'\circ\pi'$. By Proposition~\ref{prop-from_part_act_to_Lausch} it can be viewed as a $\pi''$-strict module $(A,\lambda'',\alpha'')$ for the corresponding $\pi'':S\to S''\subseteq A\rtimes S$ (in particular, $\lambda=\lambda''\circ\pi''$ and $\alpha=\alpha''\circ\pi''$). We shall show that $\ker{\pi'}=\ker{\pi''}$. 

For the inclusion $\ker{\pi''}\subseteq\ker{\pi'}$ we do not need $S$ to be $E$-unitary. Indeed, $\pi''(s)=\pi''(t)$ means $(\alpha(ss\m)\delta_s,\alpha(tt\m)\delta_t)\in\rho$ and hence by Lemma~\ref{lem-explicit_form_of_rho} there is $u\le s,t$ such that $\alpha(ss\m)\alpha(uu\m)=\alpha(ss\m)=\alpha(tt\m)$. Since $u\le s,t$ implies $uu\m\le ss\m,tt\m$, then the above equalities can be reduced to $\alpha(uu\m)=\alpha(ss\m)=\alpha(tt\m)$. These are the same as $\pi'(uu\m)=\pi'(ss\m)=\pi'(tt\m)$, because $\alpha=\alpha'\circ\pi'$ on $E(S)$ and $\alpha'$ is a bijection. But then $\pi'(s)=\pi'(ss\m)\pi'(s)=\pi'(uu\m)\pi'(s)=\pi'(uu\m s)=\pi'(u)$, as $u=uu\m s$. Similarly $\pi'(t)=\pi'(u)$. Thus, $\pi'(s)=\pi'(t)$. 

For the converse inclusion suppose that $\pi'(s)=\pi'(t)$. Then $(s,t)\in\sigma$, because $\ker{\pi'}\subseteq\sigma$. Moreover, $\pi'(ss\m)=\pi'(s)\pi'(s)\m=\pi'(t)\pi'(t)\m=\pi'(tt\m)$ and hence $\alpha(ss\m)=\alpha(tt\m)$. Thus, by Remark~\ref{rem-rho_for_E-unitary_semigr} we conclude that $(\alpha(ss\m)\delta_s,\alpha(tt\m)\delta_t)\in\rho$, i.\,e. $\pi''(s)=\pi''(t)$.

Since $\ker{\pi'}=\ker{\pi''}$, then there exists a unique isomorphism $\psi:S'\to S''$ satisfying $\psi\circ\pi'=\pi''$. The equalities $\alpha=\alpha'\circ\pi'=\alpha''\circ\pi''$ and $\lambda=\lambda'\circ\pi'=\lambda''\circ\pi''$ imply $\alpha'=\alpha''\circ\psi$ and $\lambda'=\lambda''\circ\psi$. This means that $(\id,\psi):(A,\pi')\to(A,\pi'')$ is an isomorphism in $\ESMod S$.
	\end{proof}

	\begin{cor}\label{cor-eta:Gamma(G)_to_G}
		Under the conditions of Proposition~\ref{prop-desc_of_ESMod^*(S)} if $(A,\pi')$ is standard, then there is an epimorphism of semigroups $\eta:S'\to S/\sigma$ such that $\eta\circ\pi'$ is the natural map $\sigma^\natural:S\to S/\sigma$. In particular, if $A$ is the standard $\cS G$-module from Corollary~\ref{cor-part_mod_as_epi-strict_S(G)-mod}, then $\eta$ can be viewed as an epimorphism $S'\to G$ by means of $\eta(1_{x_1}\dots 1_{x_n}1_y\delta_y)=y$.
	\end{cor}

From now on we shall work only with inverse modules.
	\begin{defn}\label{defn-InvESMod(S)}
		An epi-strict $S$-module $(A,\pi)$ will be called {\it inverse}, if $A$ is an inverse semigroup. The subcategory of inverse epi-strict $S$-modules is denoted by $\InvESMod S$.
	\end{defn}

Following Lausch~\cite{Lausch}, we are going to define a concept of a free object in $\InvESMod S$. Let us first agree on what will be a base of a free inverse epi-strict $S$-module.

In the classical case~\cite{Lausch} a base of a free $S$-module is a so-called $E(S)$-set. It is a disjoint union of sets indexed by $E(S)$. A morphism of $E(S)$-sets is a map which agrees with the partitions of these sets (i.\,e. it sends the $e$-component of one set to the subset of the $e$-component of another set for all $e\in E(S)$). Such sets appear in our situation and they will be called {\it strict} $E(S)$-sets. However, we need a more general concept of an $E(S)$-set.

	\begin{defn}\label{defn-pi-strict_E(S)-set}
		Let $S$ and $S'$ be inverse semigroups and $\pi:S\to S'$ an epimorphism. {\it A $\pi$-strict $E(S)$-set} is a pair $(T,\pi)$ where  $T$ is a strict $E(S')$-set.
	\end{defn}
In particular, $\id$-strict $E(S)$-sets can be identified with the ``ordinary'' (strict) $E(S)$-sets. As above, $\pi$ will be often omitted and a $\pi$-strict $E(S)$-set, whose $\pi$ is not specified, will be called an {\it epi-strict} $E(S)$-set. We shall also use the standard notation $T_e$ for the $e$-component of $T$ ($e\in E(S')$).

	\begin{defn}\label{defn-morph_of_pi-strict-E(S)-sets}
		Given two epi-strict $E(S)$-sets $(T',\pi':S\to S')$ and $(T'',\pi'': S\to S'')$, by {\it a morphism of epi-strict $E(S)$-sets} $(T',\pi')\to(T'',\pi'')$ we mean a pair $(\varphi, \psi)$, where $\psi:S'\to S''$ is a homomorphism of semigroups, satisfying  $\psi\circ\pi'= \pi''$, and $\varphi:T'\to T''$ is a map of sets such that $\varphi(T'_e)\subseteq T''_{\psi(e)}$ for each $e\in E(S')$.
	\end{defn}
If $S'=S''$ and $\pi'=\pi''$, then $(\varphi,\psi)$ is a morphism of $\pi$-strict $E(S)$-sets if and only if $\psi=\id$ and $\varphi$ is a classical~\cite{Lausch} morphism of (strict) $E(S')$-sets. The category of epi-strict $E(S)$-sets and their morphisms will be denoted by $\ESSet S$.

	\begin{rem}\label{rem-pi-str_mod_is_pi-str_set}
		Each inverse $\pi$-strict $S$-module is a $\pi$-strict $E(S)$-set and each morphism of inverse epi-strict $S$-modules is a morphism of epi-strict $E(S)$-sets.
	\end{rem}
	\noindent The first assertion immediately follows from the fact that an inverse $\pi'$-strict $S$-module $(A',\lambda',\alpha')$, $\pi':S\to S'$, being an inverse strict $S'$-module, is a strict $E(S')$-set (see~\cite{Lausch}). Its $e$-component is $A'_e=\{a\in A'\mid aa\m=\alpha'(e)\}$, $e\in E(S')$. Now if $(\varphi,\psi):(A',\pi')\to(A'',\pi'')$ is a morphism of inverse epi-strict $S$-modules and $a\in A'_e$ for some $e\in E(S')$, then (ii) of Definition~\ref{defn-morph_of_pi_strict_S-mod} implies 
	\[
		\varphi(a)\varphi(a)^{-1}=\varphi(aa^{-1})=\varphi(\alpha'(e))=
\alpha''(\psi(e)),
	\]
so $\varphi(a)\in A''_{\psi(e)}$.

	\begin{defn}\label{defn-free-epi-strict-S-mod}
		We say that a module $F\in\InvESMod S$ is {\it free} over a set $T\in\ESSet S$ if there is a  morphism $(\iota,\kappa):T\to F$ in $\ESSet S$ such that for any $A\in\InvESMod S$ and any morphism $(\varphi,\psi):T\to A$ in $\ESSet S$ there exists a unique morphism $(\check\varphi,\check\psi):F\to A$ in $\InvESMod S$ such that $(\varphi,\psi)=(\check\varphi,\check\psi)\circ(\iota,\kappa)$.
	\end{defn}

	We recall Lausch's construction of a free (strict) $S$-module. Given an $E(S)$-set $T$, define $F=F(T)$ to be the disjoint union of components $F_e$, where each $F_e$ is the free abelian group (written additively) over
	\[
		\{(s,t)\in S\times T\mid ss^{-1}=e,\ \ t\in T_f\mbox{ for some }f\ge s\m s\}.
	\]
The sum of the elements $(s,t)\in F_e$ and $(s',t')\in F_{e'}$ of different components is the formal sum $(e's,t)+(es',t')$ in $F_{ee'}$. The structure of an $S$-module is defined by $\lambda^F_{s'}(s,t)=(s's,t)$ and $\alpha^F(e)=0_e$, where $0_e$ is the zero of $F_e$. The embedding $\iota:T\to F$ is given by
	\begin{equation}\label{eq-emb_of_T_into_F(T)}
		T_e\ni t\mapsto(e,t)\in F_e.
	\end{equation}

	\begin{rem}\label{rem-(e,t)_generates_F(T)}
		Notice that each $(s,t)$, where $t\in T_f$, can be written as $\lambda^F_s(f,t)$, because $s\m s\le f$ and hence $sf=(ss\m s)f=s(s\m sf)=s(s\m s)=s$ in this case.
	\end{rem}

	\begin{prop}\label{prop-free_epi-strict_S-mod}
		For any epi-strict $E(S)$-set $(T,\pi:S\to S')$ there exists a free inverse epi-strict $S$-module $F(T)$ over $T$.
	\end{prop}
	\begin{proof}
		Since $(T,\pi:S\to S')$ is an $E(S')$-set in the sense of Lausch, then by~\cite[Proposition~3.1]{Lausch} there is the free Lausch's $S'$-module $(F(T),\lambda^{F(T)},\alpha^{F(T)})$ which therefore can be considered as a $\pi$-strict $S$-module. We shall show that $F(T)$ is free over $T$ in $\InvESMod S$. 

Let $\iota:T\to F(T)$ be the embedding~\eqref{eq-emb_of_T_into_F(T)}. As it was noticed after Definition~\ref{defn-morph_of_pi-strict-E(S)-sets}, the pair $(\iota,\id)$ is a morphism in $\ESSet S$. Take any module $(A,\lambda,\alpha)\in\InvESMod S$ and $(\varphi,\psi):T\to A$ in $\ESSet S$. It will be convenient to use the additive notation for $A.$ We need to find $(\check\varphi,\check\psi)$ in $\InvESMod S$, such that $\check\varphi\circ\iota=\varphi$ and $\check\psi\circ\id=\psi$. The second condition immediately gives $\check\psi=\psi$. The first one uniquely defines $\check\varphi$ as a homomorphism satisfying (iii) of De\-fi\-ni\-tion~\ref{defn-morph_of_pi_strict_S-mod}. Indeed, take $(s,t)\in F(T)_e$ with $t\in T_f$ and note by Remark~\ref{rem-(e,t)_generates_F(T)} that $\check\varphi(s,t)=\check\varphi\circ\lambda^{F(T)}_s(f,t)$,
which should be equal to $\lambda_{\psi(s)}\circ\check\varphi(f,t)=
\lambda_{\psi(s)}\circ\check\varphi\circ\iota(t)=\lambda_{\psi(s)}(\varphi(t))$. Therefore 
	\[
		\check\varphi\left(\sum a_{s,t}(s,t)\right)=\sum a_{s,t}\lambda_{\psi(s)}(\varphi(t))
	\]
and (iii) of Definition~\ref{defn-morph_of_pi_strict_S-mod} is straightforward. Moreover, since $\varphi(T_e)\subseteq A_{\psi(e)}$ for all $e\in E(S')$, we observe that
	\[
		\lambda_{\psi(s)}(\varphi(t))\in\lambda_{\psi(s)}(A_{\psi(f)})\subseteq A_{\psi(s)\psi(f)\psi(s)\m}=A_{\psi(sfs\m)}=A_{\psi(ss\m)}=A_{\psi(e)}, 
	\]
so $\check\varphi$ is automatically a homomorphism from $F(T)_e$ to $A_{\psi(e)}$ for all $e\in E(S')$. Hence, it maps the zero $\alpha^{F(T)}(e)$ of $F_e$ to the zero $\alpha(\psi(e))$ of $A_{\psi(e)}$ and thus (ii) of Definition~\ref{defn-morph_of_pi_strict_S-mod} is also true. It only remains to make sure that $\check\varphi$ is a homomorphism of semigroups (written additively). For $(s,t)\in F(T)_e$ and $(s',t')\in F(T)_{e'}$ we have
	\begin{align*}
		\check\varphi((s,t)+(s',t'))&=\check\varphi((e's,t)+(es',t'))=
\lambda_{\psi(e's)}(\varphi(t))+\lambda_{\psi(es')}(\varphi(t'))\\
		&=\lambda_{\psi(e')}(\lambda_{\psi(s)}(\varphi(t)))+\lambda_{\psi(e)}(\lambda_{\psi(s')}(\varphi(t')))\\
		&=\alpha(\psi(e'))+\lambda_{\psi(s)}(\varphi(t))+\alpha(\psi(e))+\lambda_{\psi(s')}(\varphi(t'))\\
&=(\alpha(\psi(e))+\lambda_{\psi(s)}(\varphi(t)))+(\alpha(\psi(e'))+\lambda_{\psi(s')}(\varphi(t')))\\
&=\lambda_{\psi(s)}(\varphi(t))+\lambda_{\psi(s')}(\varphi(t'))=\check\varphi(s,t)+\check\varphi(s',t').
	\end{align*}
	\end{proof}

	\begin{defn}\label{defn-pi-component}
		The full subcategory of $\InvESMod S$ formed by  the objects $(A,\pi)$ with the same $\pi:S\to S'$ will be called {\it the $\pi$-component} of $\InvESMod S$. 
	\end{defn}

	\begin{rem}\label{rem-pi-comp_is_isom_to_StrMod(S')}
		Note that the $\pi$-component of $\InvESMod S$, where $\pi:S\to S'$, is isomorphic to the category of strict inverse $S'$-modules. In particular, it is an abelian category having enough projectives (see~\cite{Lausch}).
	\end{rem}

	\begin{rem}\label{rem-F(T,pi)_belongs_to_pi-comp}
		The free inverse epi-strict module $F(T)$ over $(T,\pi:S\to S')$ belongs to the $\pi$-component of $\InvESMod S$. Moreover, it is projective in this component.
	\end{rem}

	\noindent Indeed, by Corollary 3.2 from~\cite{Lausch} $F(T)$ is projective as a strict inverse $S'$-module.

	\begin{cor}\label{cor-nonstand_epi-str_mod}
		There are non-standard epi-strict modules.
	\end{cor}
	\noindent As an example consider an inverse semigroup $S$ with $\sigma\ne S^2$ and the epimorphism $\pi:S\to 0$ to the zero semigroup. Clearly $S^2=\ker\pi\not\subseteq\sigma$. Then for any $\pi$-strict set $T$ the free inverse $\pi $-strict $S$-module $F(T)$ is not standard.

According to Propositions~\ref{prop-Exel_th_for_act_on_semigr},  \ref{prop-from_part_act_to_Lausch}, Theorem~\ref{thrm-from_part_iso_to_endo} and Remark~\ref{rem-inv_part_G-mod_is_inv_S(G)-mod}, each inverse partial $G$-module can be seen as an inverse epi-strict $\cS G$-module. 

	\begin{defn}\label{defn-free_inv_part_mod}
		Let $A$ be an inverse partial $G$-module and $T$ an epi-strict $E(\cS G)$-set. The module $A$ is said to be {\it free} over $T$ if the corresponding inverse epi-strict $\cS G$-module is free over $T$ in $\InvESMod{\cS G}$.
	\end{defn}

	We shall  now  examine the construction of a free inverse partial $G$-module in more detail. Let $M$ be a monoid and $\G:G\to M$ a partial homomorphism. Denote by $\G(G)$ the submonoid of $M$ generated by all $\G(x)$, $x\in G$. Then $\G(G)$ is an epimorphic image of $\cS G$, and every epimorphic image $\pi(\cS G)$ of $\cS G$ can be obtained this way by setting $\G=\pi\circ[\phantom{x}]$, where $[\phantom{x}]:G\to\cS G$ is the canonical partial homomorphism $x\mapsto[x]$. Notice that $\G (G)$ is an inverse monoid with $\G (x)\m = \G (x\m ).$

Write $\epsilon_x=\G(x)\G(x\m)=\pi(\varepsilon_x)$. Then one has that 
	\[
		\G(x)\epsilon_y=\epsilon_{xy}\G(x) 
	\] 
for all $x,y\in G$ (see\footnote{The cited results in~\cite{DN} are stated for a partial projective representation $\G $ of $G$ in a $K$-cancellative monoid $M.$ Notice that any monoid $M$ can be transformed into a $K$-cancellative monoid over $K={\rm GF}(2)$ by adding a zero $0_M$ to $M$ (if necessary) and setting $0_K a= 0_M, 1_K a =a,$ $a\in M.$ Then any partial homomorphism $\G :G \to M$ becomes a partial projective representation over $K$ with trivial factor set.}~\cite[Section 7]{DN}). Let $E$ be the submonoid of $\G(G)$ generated by the idempotents $\epsilon_x$. We know from~\cite[Theorem~6]{DN} that $\G(G)$ is an epimorphic image of $E*_\0 G$ where 
	\[
		\0=\0^\G=\{\0_x :  E_{x\m}\to E_x\}
	\] 
is the partial action of $G$ on $E$ which corresponds to $\G$, with $E_x = \epsilon_xE$ and $\0_x(e)=\G(x)e\G(x\m)$, $e\in E_{x\m}$. It follows that $\G(G)$ is a (not necessarily disjoint) union 
	\[
		\G(G)=\bigcup_{x\in G}E_x\G(x).
	\]   
Moreover, since we have an epimorphism  $\cS G \to \Gamma (G),$  we obtain that $E=E(\G(G))$, the set of all idempotents of $\G(G)$. Given an arbitrary element $s=e\G(x)$ ($e\in E_x$) of $\G(G)$, write $e=\epsilon_x\epsilon_{y_1}\dots \epsilon_{y_k}$ ($y_1,\dots, y_k \in G$). Then 
$ss\m=e$ and 
	\[
		s\m s=\epsilon_{x\m}\epsilon_{x\m y_1}\dots\epsilon_{x\m y_k}=\G(x\m)e\G(x)=\0_{x\m}(e).
	\]
Given a strict $E$-set $T=\bigsqcup_{e\in E}T_e$, the free inverse partial  $G$-module $F(T)$ over $T$ can be specified as follows. Taking a fixed $e\in E$ set
	\[
		\G_e=\{e\G(x)\mid x\in G, e \in E_x  \}.
	\] 
The condition $e \in E_x $  means that $e$ can be written as a product $e=\epsilon_x\epsilon_{y_1}\dots \epsilon_{y_k}$ for some $y_1,\dots,y_k\in G$. Observe  that $\G_e$ consists of all $s\in\G(G)$ with $ss\m =e$. Next, for an arbitrary $s=e\G(x)\in\G_e$ let $E(s)$ be the set of all $f\in E$ such that $f\ge s\m s$. In particular,  $E(s)$ contains  all $f\in E$ which can be written in the form $f=(\epsilon_{x\m})^{\nu}\prod_{j\in J}\epsilon_{x\m y_j}$, where $\nu\in\{0,1\}$ and  $J$ is a (possibly empty) subset of $\{1,\dots,k\}$ (if $\nu=0$ and $J=\emptyset$ we assume $f=1_M$). Write $T(s)=\bigsqcup_{f\in E(s)}T_f$. Then the $e$-component $F(T)_e$ of $F(T)$ is the free abelian group with free basis $B_e=\{(s,t)\mid s\in\G_e, t\in T(s)\}$. Denote by $1_{(e)}$ the identity element of $F(T)_e$ and write $1_x=1_{(\epsilon_x)}$, $x\in G$. The elements from $\bigsqcup_{e\in E}B_e$ will be called the canonical generators of $F(T)$. The product of canonical generators $(s,t)\in F(T)_e$ and $(s',t')\in F(T)_{e'}$ from different components is $(e's,t)(es',t')$ 
as an element of $F(T)_{ee'}$. %it was written above  

The corresponding partial action $\0^{F(T)}$ of $G$ on $F(T)$ consists of the isomorphisms $\0^{F(T)}_x:\cD_{x\m}\to\cD_x$, where 
	\[
		\cD_x=1_xF(T)=\bigsqcup_{ e \in E_x } F(T)_e,
	\]
$x\in G$, and $\0^{F(T)}_x(s,t)=(\G(x)s,t)$ for any canonical generator $(s,t)$ which is contained in $\cD_{x\m}$. If $(s,t)\in\cD_{x\m}$, then $s=\epsilon_{x\m}e\G(y)$ for some $e\in E$ and $y\in G$ such that $\epsilon_{x\m}e\epsilon_y=\epsilon_{x\m}e$. Then     
	\[
		\0^{F(T)}_x(\epsilon_{x\m}e\G(y),t)=(\0_x(\epsilon_{x\m}e)\G(xy),t).
	\]
Writing $e=\epsilon_y\epsilon_{z_1}\dots\epsilon_{z_k}$, where $z_1,\dots, z_k\in G$, we have  $\0_x(\epsilon_{x\m}e)=\epsilon_x\epsilon_{xy}\epsilon_{xz_1}\dots \epsilon_{xz_k}$. An arbitrary element $a\in\cD_{x\m}$ belongs to some component  $F(T)_e$ and can be written as a combination of canonical generators of $F(T)_e$:
	\[
		a = (s_1, t_1) ^ {n_1}\dots (s_r, t_r) ^ {n_r},
	\] 
$n_i\in\bbZ$, and one has 
	\[
		\0^{F(T)}_x(a)=(\0^{F(T)}_x(s_1, t_1))^{n_1}\dots(\0^{F(T)}_x(s_r,t_r))^{n_r}.
	\]  

	\section{Partial cohomology of groups and co\-ho\-mo\-lo\-gy of inverse semigroups}\label{sec-from_part_cohom_to_Lausch_cohom}

    Recall from~\cite{Lausch} that, given an inverse semigroup $S$ and a strict inverse $S$-module $A$, the $n$-th cohomology group of $S$ with values in $A$, denoted by $H^n_S(A)$, $n\ge 0$, is $\Ext_S^n(\bbZ_S,A)$, where $\bbZ_S$ is the ``trivial'' strict inverse $S$-module with $(\bbZ_S)_e=\{n_e\mid n\in\bbZ\}$, $n_e+m_f=(n+m)_{ef}$, $\lambda_s(n_e)=n_{ses\m}$, $\alpha(e)=0_e$. It was proved in~\cite[Theorem 2.7]{Loganathan81} that this cohomology is a particular case of the cohomology of small categories. Moreover, when $S$ is a monoid, \cite[Theorem 2.9]{Loganathan81} implies that $H^n_S(A)$ can be seen as the cohomology of the category $\mathcal D(S)$, introduced by J.~Leech in~\cite{Leech75}.
    
	\begin{defn}\label{defn-H^n(S,A,pi)}
		The cohomology groups of an inverse semigroup $S$ with va\-lues in an inverse epi-strict $S$-module $(A,\pi:S\to S')$ can be defined in the following way: $H^n(S,A)=H_{S'}^n(A)$, where on the right-hand side $A$ is meant to be a strict $S'$-module (see~\cite{Lausch}). 
	\end{defn}

Under the conditions of the above definition denote by $\Hom_\pi(-,A)$ the restriction of $\Hom(-,A)$ to the $\pi$-component of $\InvESMod S$. Then it is an additive functor from an abelian category to $\Ab$, and the above cohomology is $R^n\Hom_\pi(-,A)$ applied to $\bbZ_{S'}$. Hence, $H^n(S,A)$ can be calculated by taking an appropriate projective resolution of $\bbZ_{S'}$ in the $\pi$-component of $\InvESMod S$. It turns out that for $S=\cS G$ and $(A,\pi:S\to S')$, which induces an inverse $(A,\0)\in\pMod G$, we can construct a free resolution of $\bbZ_{S'}$ (whose terms are free in $\InvESMod S$), such that the corresponding cohomology groups with values in $(A,\pi)$ are precisely $H^n(G,A)$.

	\begin{defn}\label{defn-free_res_of_Z_S}
		Let $\G:G\to S$ be a partial homomorphism of a group $G$ in an inverse monoid $S$. For any positive integer $n$ denote by $V_n$ the strict $E(S)$-set, whose $e$-component is the (possibly empty) set of ordered $n$-tuples $(x_1,\dots,x_n)\in G^n$, such that $\epsilon_{(x_1,\dots,x_n)}=e$, where
	\begin{align}\label{eq-e_(x_1,...,x_n)}
		\epsilon_{(x_1,\dots,x_n)}=\epsilon_{x_1}\epsilon_{x_1x_2}\dots\epsilon_{x_1\dots x_n}=\G(x_1)\dots\G(x_n)\G(x_n)\m\dots\G(x_1)\m\in E(S).
	\end{align}
	
For $n=0$ define $(V_n)_e$ to be the one-element set $\{(\phantom{x})\}$ if $e=1_S$, and $\emptyset$ otherwise. The free $S$-module over $V_n$ will be denoted by $R_n$.
	\end{defn}
 
	\begin{lem}\label{lem-Hom_S(R_n,A)=C^n(G,theta,A)}
		Let $(A,\0)$ be an inverse partial $G$-module and the pair $(A,\lambda,\alpha),\pi:\cS G\to S$ a standard inverse $\pi$-strict $\cS G$-module inducing $(A,\0)$ in the sense of~\eqref{eq-lambda_and_alpha-for-S}. Set $\Gamma=\pi\circ[\phantom{x}]:G\to S$ and consider $R_n=R_n(\G)$ as an element of the $\pi$-component of $\InvESMod{\cS G}$. Then the abelian group $\Hom_\pi(R_n,A)$ is isomorphic to $C^n(G,A)$.
	\end{lem}
	\begin{proof}
		Since $R_n$ is the free $S$-module over $V_n=V_n(\G)$, each morphism from $R_n$ to $A$ is fully determined by its values on $V_n$. So, $\Hom_\pi(R_n,A)$ can be identified with the set of morphisms of $\pi$-strict $E(\cS G)$-sets from $V_n$ to $A$. A pair $(\varphi,\psi)$ is such a morphism if and only if $\psi=\id_S$ and  $\varphi(v)\in A_e$ for all $v\in(V_n)_e$, $e\in E(S)$. If $n=0$, then $\varphi$ is identified with the image of $(\phantom{x})\in(V_n)_{1_S}$, which should belong to $A_{1_S}$, i.\,e. be invertible with respect to
	\[
		\alpha(1_S)=\alpha\circ\pi(1_{\cS G})=\alpha^\0(1_{\cS G})=1_A.
	\] 
Thus, $\Hom_\pi(R_0,A)\cong\cU A=C^0(G,A)$. For $n>0$ the function $\varphi$ can be viewed as a map $G^n\to A$, such that $\varphi(x_1,\dots,x_n)\in A_e$, where $e=\epsilon_{(x_1,\dots,x_n)}$. By definition $A_e$ consists of $a\in A$, which are invertible with respect to 
	\[
		\alpha(e)=\alpha\circ\pi(\e_{x_1}\e_{x_1x_2}\dots\e_{x_1\dots x_n})=1_{x_1}1_{x_1x_2}\dots 1_{x_1\dots x_n}.
	\]
So, $A_e=\cU{A_{(x_1,\dots,x_n)}}$ and hence $\Hom_\pi(R_n,A)\cong C^n(G,A)$.
	\end{proof}

	\begin{defn}\label{defn-partial_n}
		Under the conditions of Definition~\ref{defn-free_res_of_Z_S} for all $n\ge 1$ define the  $S$-module morphisms $\partial_n:R_n\to R_{n-1}$ by
	\begin{align*}
		\partial_n(x_1,\dots,x_n)&=(\G(x_1)\epsilon_{(x_2,\dots,x_n)},(x_2,\dots,x_n))\\
	&+\sum_{i=1}^{n-1}(-1)^i(\epsilon_{(x_1,\dots,x_n)},(x_1,\dots,x_ix_{i+1},\dots,x_n))\\
	&+(-1)^n(\epsilon_{(x_1,\dots,x_n)},(x_1,\dots,x_{n-1})), n>1,\\
		\partial_1(x)&=(\G(x),(\phantom{x}))-(\epsilon_x,(\phantom{x})),
	\end{align*}
and $\partial_0:R_0\to\bbZ_S$ by $\partial_0(\phantom{x})=1_{1_S}$.
	\end{defn}

	\begin{lem}\label{lem-delta^nf=fpartial_n}
		Under the conditions of Lemma~\ref{lem-Hom_S(R_n,A)=C^n(G,theta,A)} for any $n\ge 0$ and for arbitrary $f\in C^n(G,A)$ we have $\delta^nf=f\circ\partial_{n+1}$, where $f$ and $\delta^nf$ are considered as homomorphisms of $\pi$-strict $\cS G$-modules $R_n\to A$ and $R_{n+1}\to A$ respectively.
	\end{lem}
	\begin{proof}
		 Denote by $\iota_n$ the natural embedding $V_n\to R_n$ given in~\eqref{eq-emb_of_T_into_F(T)}. If $n=0$, then $f\in C^0(G,A)$, considered as $f:R_0\to A$, is identified with $f(\iota_0(\phantom{x}))=a\in\cU A$, where $\iota_0(\phantom{x})=(1_S,(\phantom{x}))$. Therefore, for any $x\in G$:
	\begin{align*}
		f(\partial_1(x))&=f(\G(x),(\phantom{x}))f(\epsilon_x,(\phantom{x}))\m
=\lambda_{\G(x)}(a)\lambda_{\epsilon_x}(a)\m=\lambda_{\pi([x])}(a)\lambda_{\pi(\e_x)}(a)\m\\
&=\0_x(1_{x\m}a)(a1_x)\m=\0_x(1_{x\m}a)a\m 1_x=\0_x(1_{x\m}a)a\m=(\delta^0a)(x).
	\end{align*}
Now let $n>0$ and $f\in C^n(G,A)$. Considering $f$ as the morphism $R_n\to A$, we notice that
	\[
		f(x_1,\dots,x_n)=f(\iota_n(x_1,\dots,x_n))=f(\epsilon_{(x_1,\dots,x_n)},(x_1,\dots,x_n)).
	\]
Therefore,
	\begin{align*}
		f(\partial_{n+1}(x_1,\dots,x_{n+1}))&=\lambda_{\G(x_1)}(f(x_2,\dots,x_{n+1}))\\
&\prod_{i=1}^n\lambda_{\epsilon_{(x_1,\dots ,x_{n+1})}}(f(x_1,\dots,x_ix_{i+1},\dots,x_{n+1}))^{(-1)^i}\\
&\lambda_{\epsilon_{(x_1, \dots , x_{n+1})}}(f(x_1,\dots,x_n))^{(-1)^{n+1}}.
	\end{align*}
As we have seen above, $\lambda_{\Gamma(x)}(a)=\0_x(1_{x\m}a)$ and $\lambda_{\epsilon_{x}}(a)=a1_x$. Thus,
	\[
		f(\partial_{n+1}(x_1,\dots,x_{n+1}))=(\delta^nf)(x_1,\dots,x_{n+1})1_{x_1}1_{x_1x_2}\dots 1_{x_1\dots x_{n+1}}
	=(\delta^nf)(x_1,\dots,x_{n+1}),
	\]
because $(\delta^nf)(x_1,\dots,x_{n+1})\in A_{(x_1,\dots,x_{n+1})}$.
	\end{proof}

	\begin{cor}\label{cor-partial^2=0}
		Under the conditions of Lemma~\ref{lem-Hom_S(R_n,A)=C^n(G,theta,A)} 
for any $n\ge 2$ the composition $\partial_{n-1}\circ\partial_n$ is the zero morphism from $R_n$ to $R_{n-2}$ in the $\pi$-component of $\InvESMod{\cS G}$.
	\end{cor}
	\noindent Indeed, by assumption the module $(A,\pi)$ is standard. Since $\cS G$ is $E$-unitary, by Proposition~\ref{prop-desc_of_ESMod^*(S)} any module from the $\pi$-component of the ca\-te\-go\-ry $\InvESMod{\cS G}$ is standard, in particular all $R_n$ have this property. Thus, without loss of generality we may assume that $(R_n,\pi)$ is obtained from some inverse partial $G$-module $(R_n,\0_n)$ for all $n\ge 0$. Now taking $A'=R_{n-2}$ and $f=\id_{R_{n-2}}\in\Hom_\pi(R_{n-2},A')$ by Lemma~\ref{lem-delta^nf=fpartial_n} and Proposition~\ref{prop-coboundary_hom} we see that
	\[
		\partial_{n-1}\circ\partial_n=f\circ\partial_{n-1}\circ\partial_n=
\delta^{n-1}\delta^{n-2}f=e_n.
	\]
In view of Lemma~\ref{lem-Hom_S(R_n,A)=C^n(G,theta,A)} the cochain $e_n$ can be seen as the morphism from $R_n$ to $R_{n-2}$ which maps each element of $(V_n)_e$ to the zero of $(R_{n-2})_e$, $e\in E(S)$. So, $e_n=0\in\Hom_\pi(R_n,R_{n-2})$.

	\begin{defn}\label{defn-sigma_n}
		Under the conditions of Lemma~\ref{lem-Hom_S(R_n,A)=C^n(G,theta,A)} let $\eta:S\to G$ be as in Corollary~\ref{cor-eta:Gamma(G)_to_G} and for all $n\ge 0$ define the morphism  $\sigma_n:R_n\to R_{n+1}$ of $\pi$-strict $E(\cS G)$-sets as follows. On each $e$-component of $R_n$ it is the homomorphism of abelian groups $(R_n)_e\to(R_{n+1})_e$ given by
	\begin{align*}
	    \sigma_0(s,(\phantom{x}))&=(ss\m,(\eta(s))),\\
	    \sigma_n(s,(x_1,\dots,x_n))&=(ss\m,(\eta(s),x_1,\dots,x_n)),\ \ n>0.
	\end{align*}
We shall also define $\sigma_{-1}:\bbZ_S\to R_0$ to be the homomorphism $(\bbZ_S)_e\to (R_0)_e$, such that $\sigma_{-1}(1_e)=(e,(\phantom{x}))$, $e\in E(S)$.
	\end{defn}

The fact that $\sigma_n$ is well-defined will follow from the next lemma.

	\begin{lem}\label{lem-properties_of_eta}
		Under the conditions of Definition~\ref{defn-sigma_n} for all $x \in G$ and $s\in S$
	\begin{enumerate}[(i)]
		\item $\eta(\G(x))=x;$
		\item $ss\m\G(\eta(s))=s;$
		\item $s\G(\eta(s)\m)=ss\m .$
	\end{enumerate}
	\end{lem}
	\begin{proof}
		Notice that $\eta\circ\pi$ is the natural epimorphism $\cS G\to G,$ so that $ \eta(\G(x))=\eta(\pi([x]))=x,$ which is (i). If $s=\pi(\e_{x_1}\dots\e_{x_n}[y])$, then $\eta (s)=y $ and  $\G(\eta(s))=\G (y) = \pi([y])$, so (ii) and (iii) follow from  equalities in $\cS G$, which can be easily verified. 
	\end{proof}

    In particular, it follows from (ii) that $s\le\G(\eta(s))$, so $ss\m\le\epsilon_{\eta(s)}$ and thus $(ss\m,(\eta(s)))\in R_1$. Now taking $(s,(x_1,\dots,x_n))\in R_n$, we show that $(ss\m,(\eta(s),x_1,\dots,x_n))\in R_{n+1}$, i.\,e. $ss\m\le \epsilon_{(\eta(s),x_1,\dots,x_n)}.$  Using (ii) and (iii) and the fact that $s\epsilon_{(x_1,\dots,x_n)}=s$ (this follows from $s\m s\le \epsilon_{(x_1,\dots,x_n)}$), we see that
	\begin{align*}
		ss\m \epsilon_{(\eta(s),x_1,\dots,x_n)}=ss\m\G(\eta(s))\epsilon_{(x_1,\dots,x_n)}\G(\eta(s)\m)&=
s\epsilon_{(x_1,\dots,x_n)}\G(\eta(s)\m)\\
	&=s\G(\eta(s)\m)=ss\m.
	\end{align*}
	
	\begin{lem}\label{lem-sigma_n_is_contr_homotopy}
		Under the conditions of Lemma~\ref{lem-Hom_S(R_n,A)=C^n(G,theta,A)} we have 
		\begin{align}\label{eq-contr-homotopy}
		    \partial_{n+1}\circ\sigma_n+\sigma_{n-1}\circ\partial_n=\id_{R_n},\ \ n\ge 0.
		\end{align}
	\end{lem}
	\begin{proof}
		Let $n=0$. It is sufficient to verify~\eqref{eq-contr-homotopy} on an arbitrary generator $(s,(\phantom{x}))$ of $(R_0)_{ss\m}.$ Using (ii) of Lemma~\ref{lem-properties_of_eta} and the inequality $ss\m\le\epsilon_{\eta(s)}$ explained above, we have:
	\begin{align*}
		(\partial_1\circ\sigma_0+\sigma_{-1}\circ\partial_0)(s,(\phantom{x}))&=
		\partial_1(ss\m,(\eta(s)))+\sigma_{-1}(\lambda_s(1_{1_S}))\\
		&=(ss\m\G(\eta(s)),(\phantom{x}))-(ss\m \epsilon_{\eta(s)},(\phantom{x}))+\sigma_{-1}(1_{ss\m})\\
		&=(s,(\phantom{x}))-(ss\m,(\phantom{x}))+(ss\m,(\phantom{x}))\\
        &=(s,(\phantom{x})).
	\end{align*}

	Let $n>0$. For a generator $(s,(x_1,\dots,x_n))$ of $(R_n)_{ss\m}$, taking into account the facts that $s\epsilon_{(x_1,\dots,x_n)}=s$, $ss\m\le \epsilon_{(\eta(s),x_1,\dots,x_n)}$ established above and (i)--(ii) of Lemma~\ref{lem-properties_of_eta}, we calculate 
	\begin{align}\label{eq-partial_(n+1)sigma_n}
		\partial_{n+1}\circ\sigma_n(s,(x_1,\dots,x_n))&=\partial_{n+1}(ss\m,(\eta(s),x_1,\dots,x_n))\notag\\
	&=(ss\m\G(\eta(s))\epsilon_{(x_1,\dots,x_n)},(x_1,\dots,x_n))\notag\\
	&-(ss\m \epsilon_{(\eta(s),x_1,\dots,x_n)},(\eta(s)x_1,x_2,\dots,x_n))\notag\\
	&+\sum_{i=1}^{n-1}(-1)^{i+1}(ss\m \epsilon_{(\eta(s),x_1,\dots,x_n)},(\eta(s),x_1,\dots,x_ix_{i+1},\dots,x_n))\notag\\
	&+(-1)^{n+1}(ss\m \epsilon_{(\eta(s),x_1,\dots,x_n)},(\eta(s),x_1,\dots,x_{n-1}))\notag\\
	&=(s,(x_1,\dots,x_n))-
	(ss\m,(\eta(s)x_1,x_2,\dots,x_n))\notag\\
	&+\sum_{i=1}^{n-1}(-1)^{i+1}(ss\m,(\eta(s),x_1,\dots,x_ix_{i+1},\dots,x_n))\notag\\
	&+(-1)^{n+1}(ss\m,(\eta(s),x_1,\dots,x_{n-1})).
	\end{align}
and
	\begin{align}\label{eq-sigma_(n-1)partial_n}
		\sigma_{n-1}\circ\partial_n(s,(x_1,\dots,x_n))&=
		\sigma_{n-1}(s\G(x_1)\epsilon_{(x_2,\dots,x_n)},(x_2,\dots,x_n))\notag\\
		&+\sum_{i=1}^{n-1}(-1)^i\sigma_{n-1}(s\epsilon_{(x_1,\dots,x_n)},(x_1,\dots,x_ix_{i+1},\dots,x_n))\notag\\
		&+(-1)^n\sigma_{n-1}(s\epsilon_{(x_1,\dots,x_n)},(x_1,\dots,x_{n-1}))\notag\\
		&=(ss\m,(\eta(s\G(x_1)\epsilon_{(x_2,\dots,x_n)}),x_2,\dots,x_n))\notag\\
		&+\sum_{i=1}^{n-1}(-1)^i(ss\m,(\eta(s),x_1,\dots,x_ix_{i+1},\dots,x_n))\notag\\
		&+(-1)^n(ss\m,(\eta(s),x_1,\dots,x_{n-1}))\notag\\
		&=(ss\m,(\eta(s)x_1,x_2,\dots,x_n))\notag\\
		&+\sum_{i=1}^{n-1}(-1)^i(ss\m,(\eta(s),x_1,\dots,x_ix_{i+1},\dots,x_n))\notag\\
		&+(-1)^n(ss\m,(\eta(s),x_1,\dots,x_{n-1})).
	\end{align}

Adding~\eqref{eq-partial_(n+1)sigma_n} and~\eqref{eq-sigma_(n-1)partial_n} we get $(s,(x_1,\dots,x_n))$, as desired.
	\end{proof}

We recall from~\cite{Lausch} that, given a morphism of (additive) strict inverse $S$-modules $\phi:(A,\lambda,\alpha)\to(A',\lambda',\alpha')$, the kernel $\ker\phi$ is the submodule of $A$ with $(\ker\phi)_e=\phi\m(\alpha'(e))$, $e\in E(S)$.

	\begin{cor}\label{cor-ker_partial_n_subset_im_partial_(n-1)}
		Under the conditions of Lemma~\ref{lem-Hom_S(R_n,A)=C^n(G,theta,A)} we have $\ker{\partial_n}\subseteq\im{\partial_{n+1}}$ for all $n\ge 0$.
	\end{cor}
	\noindent Indeed, if $\partial_0(r)=0_e$ for some $r\in(R_0)_e$, $e\in E(S),$ then by Lemma~\ref{lem-sigma_n_is_contr_homotopy} with $n=0$ and the fact that $\sigma_{-1}|_{(\bbZ_S)_e}$ is a homomorphism of abelian groups $(\bbZ_S)_e\to (R_0)_e$, we see that $r=\partial_1(\sigma_0(r))+\alpha^{R_0}(e)$. Since $\partial_1\circ\sigma_0$ is a morphism of $\pi$-strict $E(\cS G)$-sets (see Remark~\ref{rem-pi-str_mod_is_pi-str_set} and Definition~\ref{defn-sigma_n}), $\partial_1(\sigma_0(r))\in(R_0)_e$, so $\partial_1(\sigma_0(r))+\alpha^{R_0}(e)=\partial_1(\sigma_0(r))$, whence $r=\partial_1(\sigma_0(r))$. The inclusion for $n>0$ can be similarly obtained from Lemma~\ref{lem-sigma_n_is_contr_homotopy}. 

	\begin{thrm}\label{thrm-R_n_if_free_res_of_Z_S}
		Under the conditions of Lemma~\ref{lem-Hom_S(R_n,A)=C^n(G,theta,A)} the sequence
	\[
		\dots\overset{\partial_{n+1}}{\to}R_n\overset{\partial_n}{\to}\dots
\overset{\partial_2}{\to}R_1\overset{\partial_1}{\to}R_0\overset{\partial_0}{\to}\bbZ_S\to 0,
	\]
where $0$ is the zero of the $\pi$-strict component of $\InvESMod{\cS G}$, is a free resolution of $\bbZ_S$ whose cohomology groups with values in $A$ are isomorphic to $H^n(G,A)$.
	\end{thrm}
	\begin{proof}
		In view of Lemmas~\ref{lem-Hom_S(R_n,A)=C^n(G,theta,A)},~\ref{lem-delta^nf=fpartial_n} and Corollaries~\ref{cor-partial^2=0},~\ref{cor-ker_partial_n_subset_im_partial_(n-1)} it only remains to prove the inclusion $\im{\partial_1}\subseteq\ker\partial_0$ (the exactness in $\bbZ_S$ is obvious, because $n_e=\partial_0(n(e,(\phantom{x})))$ for all $n\in\bbZ $ and $e\in E(S)$). This can be done by direct computation of $\partial_0\circ\partial_1$. For any $x\in G$ the image $\partial_0(\partial_1(x))$ is 
	\begin{align*}
		\partial_0(\G(x),(\phantom{x}))-\partial_0(\epsilon_x,(\phantom{x}))
=\lambda_{\G(x)}(1_{1_S})-\lambda_{\epsilon_x}(1_{1_S})
=1_{\G(x)\G(x)\m}-1_{\epsilon_x}=0_{\epsilon_x}.
	\end{align*}
	\end{proof}

	\begin{cor}\label{cor-H^n(G,theta,A)_isom_H^n_S(A)}
		Under the conditions of Lemma~\ref{lem-Hom_S(R_n,A)=C^n(G,theta,A)} we have $H^n(G,A)\cong H^n(\cS G,A)$ for all $n\ge 0$.
	\end{cor}
	
	\section{Some remarks on max-generated $F$-inverse monoids}\label{sec-max-generated}
	
	In this section we  specify the class of inverse semigroups $S$ admitting an epimorphism $\pi:\cS G\to S$ with $\ker\pi\subseteq\sigma$. Recall from~\cite[p. 202]{Lawson} that a semigroup is said to be {\it $F$-inverse} if each its $\sigma$-class has a unique maximum element. By~\cite[Proposition 7.1.3]{Lawson} any $F$-inverse semigroup is an $E$-unitary inverse monoid. An $F$-inverse monoid $S$ is called {\it max-generated}~\cite[p. 196]{Lawson2002} if the maximum elements of its $\sigma$-classes generate $S$. For example, $\cS G$ is a max-generated $F$-inverse monoid, as $\cS G=\langle[x]\mid x\in G\rangle$.
	
	It is well-known (see~\cite[p. 63]{Lawson}) that any homomorphism $\varphi:S\to T$ of inverse semigroups induces the homomorphism of groups $\tilde\varphi:S/\sigma_S\to T/\sigma_T$ by means of $\tilde\varphi(\sigma_S(s))=\sigma_T(\varphi(s))$, because
\begin{align}\label{eq-morphism-of-group-images}
 \varphi(\sigma_S(s))\subseteq\sigma_T(\varphi(s)). 
\end{align}
Moreover, in the case we are interested in $\tilde\varphi$ becomes an isomorphism, and~\eqref{eq-morphism-of-group-images} becomes an equality, as the next lemma shows.
	
\begin{lem}\label{lem-image-of-sigma-class}
 For an epimorphism $\pi:S\to T$ of inverse semigroups the following are equivalent:
 \begin{enumerate}
  \item $\ker\pi\subseteq\sigma_S$;
  \item $\tilde\pi:S/\sigma_S\to T/\sigma_T$ is an isomorphism.
 \end{enumerate}
 Moreover, in this case $\pi(\sigma_S(s))=\sigma_T(\pi(s))$ for all $s\in S$.
\end{lem}
\begin{proof}
 Clearly, $\tilde\pi$ is surjective, as $\pi$ is an epimorphism. So, we need to prove that the inclusion (i) is equivalent to the injectivity of $\tilde\pi$. 
 
 Let $\ker\pi\subseteq\sigma_S$. Suppose that $\tilde\pi(\sigma_S(s))=\tilde\pi(\sigma_S(s'))$ for some $s,s'\in S$. By the definition of $\tilde\pi$ this means that $(\pi(s),\pi(s'))\in\sigma_T$. Therefore, $f\pi(s)=f\pi(s')$ for some $f\in E(T)$. But $f=\pi(e)$ for $e\in E(S)$, since $\pi$ is an epimorphism. Hence, $\pi(es)=\pi(es')$, i.\,e. $(es,es')\in\ker\pi$. Using (i), we conclude that $(es,es')\in\sigma_S$, and thus $(s,s')\in\sigma_S$, as $es\le s$ and $es'\le s$. 
 
 Conversely, assume that $\tilde\pi$ is injective. If $\pi(s)=\pi(s')$, then obviously $\tilde\pi(\sigma_S(s))=\tilde\pi(\sigma_S(s'))$. By injectivity of $\tilde\pi$ we have $(s,s')\in\sigma_S$.
 
 As to $\pi(\sigma_S(s))=\sigma_T(\pi(s))$, we need only establish the inclusion $\sigma_T(\pi(s))\subseteq\pi(\sigma_S(s))$. Let $t\in\sigma_T(\pi(s))$. Since $\pi$ is surjective, $t=\pi(s')$ for some $s'\in S$. So, $(\pi(s),\pi(s'))\in\sigma_T$, which implies $(s,s')\in\sigma_S$ by injectivity of $\tilde\pi$, as above. Thus, $t\in\pi(\sigma_S(s))$.
\end{proof}

\begin{cor}\label{cor-standard-image-of-max-generated}
 Assuming (i) of Lemma~\ref{lem-image-of-sigma-class}, if $S$ is $F$-inverse, then $T$ is $F$-inverse. Moreover, if $S$ is max-generated, then $T$ is max-generated.
\end{cor}
\noindent Indeed, $\pi$, being a homomorphism, agrees with the natural partial order. So, if the class $\sigma_S(s)$ has a unique maximum $\max\sigma_S(s)$, then $\pi(\sigma_S(s))$ has a unique maximum $\pi(\max\sigma_S(s))$. But since $\pi$ is surjective, any $\sigma_T$-class has the form $\sigma_T(\pi(s))=\pi(\sigma_S(s))$ for some $s\in S$. Now if the elements $\max\sigma_S(s)$, $s\in S$, generate $S$, then their images $\pi(\max\sigma_S(s))=\max\pi(\sigma_S(s))$ generate $T$.
	 
\begin{lem}\label{lem-max-generated-image-of-S(G)}
 Let $S$ be a max-generated $F$-inverse monoid. Then there is a (unique up to an isomorphism) group $G$ with a (unique) epimorphism $\pi:\cS G\to S$, such that $\ker\pi\subseteq\sigma_{\cS G}$.
\end{lem}
\begin{proof}
 Setting $G=S/\sigma_S$, by~\cite[Corollary 3]{Szendrei89} we find a (unique) homomorphism $\pi:\cS G\to S$, such that $\pi([x])=\max x$ for any $x\in G$ and 
 \begin{align}\label{eq-szendrei}
  \sigma_S^\natural\circ\pi=\sigma_{\cS G}^\natural:\cS G\to G.
 \end{align}
 Since $S$ is max-generated, it follows that $\pi$ is an epimorphism. Moreover,~\eqref{eq-szendrei} yields that $\ker\pi\subseteq\sigma_{\cS G}$. The uniqueness of $G$ (up to an isomorphism) follows from (ii) of Lemma~\ref{lem-image-of-sigma-class}.
\end{proof}
	 
As a consequence of Corollary~\ref{cor-standard-image-of-max-generated} and Lemma~\ref{lem-max-generated-image-of-S(G)} we obtain the next.
\begin{thrm}\label{thrm-max-generated}
 For an inverse semigroup $S$ the following are equivalent:
 \begin{enumerate}
  \item there exists a (unique up to an isomorphism) group $G$ admitting a (unique) epimorphism $\pi:\cS G\to S$ with $\ker\pi\subseteq\sigma_{\cS G}$;
  \item $S$ is a max-generated $F$-inverse monoid.
 \end{enumerate}
\end{thrm}
	 It follows that the cohomology of a max-generated $F$-inverse monoid can be seen as the partial group cohomolgy.
	 \begin{cor}\label{cor-H^n_S(A)-as-H^n(G,A)}
	  Let $S$ be a max-generated $F$-inverse monoid and $A$ a strict inverse $S$-module. Then there are a group $G$ (one may take $G=S/\sigma$) and a structure of an inverse partial $G$-module on $A$ such that $H^n_S(A)\cong H^n(G,A)$ for all $n\ge 0$.
	 \end{cor}

	 \noindent Indeed, using Theorem~\ref{thrm-max-generated} we find a group $G$ and an epimorphism $\pi:\cS G\to S$ with $\ker\pi\subseteq\sigma_{\cS G}$. Then by Definition~\ref{defn-H^n(S,A,pi)} we have $H^n_S(A)=H^n(\cS G,A)$, where on the right-hand side $A$ is meant to be a standard $\pi$-strict inverse $\cS G$-module. It remains to apply Corollary~\ref{cor-H^n(G,theta,A)_isom_H^n_S(A)} to the inverse partial $G$-module induced by $(A,\pi)$.\\
	 
	  Finally, we give an interpretation of a result by Loganathan (see~\cite[Proposition 3.6]{Loganathan81}) in terms of partial group cohomology. Let $S$ be an inverse semigroup, $G$ a group and $\phi:S\to G$ a homomorphism. Using the idea from~\cite[p. 384]{Loganathan81}, with any (global) $G$-module $A$ one can associate a strict inverse $S$-module $(\hat A,\lambda,\alpha)$ in the following way: 
	  \begin{align*}
	   \hat A&=\bigsqcup_{e\in E(S)}A_e, \mbox{ where } A_e=\{a_e\mid a\in A\}\cong A,\ \ a_eb_f=(a+b)_{ef},\\
	   \lambda_s(a_e)&=(\phi(s)a)_{ses\m},\ \ \alpha(e)=0_e.
	  \end{align*}
   Indeed, 
   \begin{align*}
    \lambda_s(a_eb_f)&=\lambda_s((a+b)_{ef})=(\phi(s)(a+b))_{sefs\m}=(\phi(s)a+\phi(s)b)_{ses\m\cdot sfs\m}\\
    &=(\phi(s)a)_{ses\m}(\phi(s)b)_{sfs\m}=\lambda_s(a_e)\lambda_s(b_f),\\
    \lambda_s\circ\lambda_t(a_e)&=\lambda_s((\phi(t)a)_{tet\m})=(\phi(s)\phi(t)a)_{stet\m s\m}=(\phi(st)a)_{ste(st)\m}=\lambda_{st}(a_e),\\
    \lambda_s(\alpha(e))&=\lambda_s(0_e)=(\phi(s)0)_{ses\m}=0_{ses\m}=\alpha(ses\m),\\
    \lambda_e(a_f)&=(\phi(e)a)_{ef}=(1_Ga)_{ef}=a_{ef}=0_ea_f=\alpha(e)a_f.
   \end{align*} 
   
   \begin{cor}\label{cor-H^n(G,A)-cong-H^n(G,hat A)}
    Let $S$ be a max-generated $F$-inverse monoid, $G=S/\sigma$, $\phi=\sigma^\natural:S\to G$, $A$ a $G$-module and $\hat A$ the corresponding strict inverse $S$-module. Then there is a structure of an inverse partial $G$-module on $\hat A$, such that $H^n(G,A)\cong H^n(G,\hat A)$ for all $n\ge 0$.
   \end{cor}
   \noindent For it was proved in Theorem 2.7 and Proposition 3.6 from \cite{Loganathan81} that, in our notations, $H^n(G,A)\cong H^n_S(\hat A)$ for all $n\ge 0$. The result now follows from Corollary~\ref{cor-H^n_S(A)-as-H^n(G,A)}.\\
   
   We shall specify this isomorphism in detail. Let $\pi$ be the epimorphism $\cS G\to S$ from Theorem~\ref{thrm-max-generated} and $\G=\pi\circ[\phantom{x}]$. Using~\eqref{eq-lambda_and_alpha-for-S} we define a structure $\0$ of an inverse partial $G$-module on $\hat A$. Then any $f\in C^n(G,A)$ can be seen as $\hat f: G^n\to\hat A$, where $\hat f(x_1,\dots,x_n)=f(x_1,\dots,x_n)_{\epsilon_{(x_1,\dots,x_n)}}\in A_{\epsilon_{(x_1,\dots,x_n)}}$ with $\epsilon_{(x_1,\dots,x_n)}$ given by~\eqref{eq-e_(x_1,...,x_n)}. Notice that $A_{\epsilon_{(x_1,\dots,x_n)}}=\cU{0_{\epsilon_{(x_1,\dots,x_n)}}\hat A}=\cU{\alpha(\epsilon_{(x_1,\dots,x_n)})\hat A}=\cU{1_{(x_1,\dots,x_n)}\hat A}$. Hence $\hat f\in C^n(G,\hat A)$ and the map 
   $$
    C^n(G,A)\ni f\mapsto\hat f\in C^n(G,\hat A)
   $$
   is an isomorphism of abelian groups, whose inverse is obtained by removing the index $\epsilon_{(x_1,\dots,x_n)}$ from $f(x_1,\dots,x_n)_{\epsilon_{(x_1,\dots,x_n)}}$. Moreover, using~\eqref{eq-lambda_and_alpha-for-S} and~\eqref{eq-szendrei} we have
	  \begin{align*}
	   (\delta^n\hat f)(x_1,\dots,x_{n+1})&=\0_{x_1}(1_{x_1\m}\hat f(x_2,\dots,x_{n+1}))\\
	   &\prod_{i=1}^n\hat f(x_1,\dots, x_ix_{i+1},\dots,x_{n+1})^{(-1)^i}\\
	   &\hat f(x_1,\dots,x_n)^{(-1)^{n+1}}\\
	   &=\lambda_{\G(x_1)}(f(x_2,\dots,x_{n+1})_{\epsilon_{(x_2,\dots,x_{n+1})}})\\
	   &\prod_{i=1}^nf(x_1,\dots, x_ix_{i+1},\dots,x_{n+1})_{\epsilon_{(x_1,\dots, x_ix_{i+1},\dots,x_{n+1})}}^{(-1)^i}\\
	   &f(x_1,\dots,x_n)_{\epsilon_{(x_1,\dots,x_n)}}^{(-1)^{n+1}}\\
	   &=(x_1f(x_2,\dots,x_{n+1}))_{\G(x_1)\epsilon_{(x_2,\dots,x_{n+1})}\G(x_1)\m}\\
	   &\prod_{i=1}^nf(x_1,\dots, x_ix_{i+1},\dots,x_{n+1})_{\epsilon_{(x_1,\dots, x_ix_{i+1},\dots,x_{n+1})}}^{(-1)^i}\\
	   &f(x_1,\dots,x_n)_{\epsilon_{(x_1,\dots,x_n)}}^{(-1)^{n+1}}.
	  \end{align*}
Since $\G(x_1)\epsilon_{(x_2,\dots,x_{n+1})}\G(x_1)\m=\epsilon_{(x_1,\dots,x_{n+1})}$, the latter product equals 
$$
(\delta^nf)(x_1,\dots,x_{n+1})_{\epsilon_{(x_1,\dots,x_{n+1})}}=(\widehat{\delta^nf})(x_1,\dots,x_{n+1}).
$$
Thus, $f\leftrightarrow\hat f$ is an isomorphism of cochain complexes, so it induces an isomorphism of the corresponding cohomology groups.

\section*{Acknowledgements}

The authors thank the referee for the valuable remarks which helped to improve the article.




\bibliography{bibl-pact}{}

\begin{thebibliography}{10}

\bibitem{AraEKa}
{\sc Ara, P., Exel, R., and Katsura, T.}
\newblock Dynamical systems of type $(m,n)$ and their {$C^*$}-algebras.
\newblock {\em Ergodic Theory Dynam. Systems 33}, 1 (2013), 1291--1325.

\bibitem{Birget-Rhodes84}
{\sc Birget, J.-C., and Rhodes, J.}
\newblock Almost finite expansions of arbitrary semigroups.
\newblock {\em J. Pure Appl. Algebra 32}, 3 (1984), 239--287.

\bibitem{Birget-Rhodes89}
{\sc Birget, J.-C., and Rhodes, J.}
\newblock Group theory via global semigroup theory.
\newblock {\em J. Algebra 120}, 2 (1989), 284--300.

\bibitem{Clifford-Preston-2}
{\sc Clifford, A., and Preston, G.}
\newblock {\em The algebraic theory of semigroups}, vol.~2 of {\em Math.
  Surveys and Monographs 7}.
\newblock Amer. Math. Soc., Providence, Rhode Island, 1967.

\bibitem{D}
{\sc Dokuchaev, M.}
\newblock Partial actions: a survey.
\newblock {\em Contemp. Math. 537\/} (2011), 173--184.

\bibitem{DE}
{\sc Dokuchaev, M., and Exel, R.}
\newblock Associativity of crossed products by partial actions, enveloping
  actions and partial representations.
\newblock {\em Trans. Amer. Math. Soc. 357}, 5 (2005), 1931--1952.

\bibitem{DEP}
{\sc Dokuchaev, M., Exel, R., and Piccione, P.}
\newblock Partial representations and partial group algebras.
\newblock {\em J. Algebra 226}, 1 (2000), 251--268.

\bibitem{DES1}
{\sc Dokuchaev, M., Exel, R., and Sim{\'o}n, J.~J.}
\newblock Crossed products by twisted partial actions and graded algebras.
\newblock {\em J. Algebra 320}, 8 (2008), 3278--3310.

\bibitem{DES2}
{\sc Dokuchaev, M., Exel, R., and Sim{\'o}n, J.~J.}
\newblock Globalization of twisted partial actions.
\newblock {\em Trans. Amer. Math. Soc. 362}, 8 (2010), 4137--4160.

\bibitem{DFP}
{\sc Dokuchaev, M., Ferrero, M., and Paques, A.}
\newblock Partial actions and {G}alois theory.
\newblock {\em J. Pure Appl. Algebra 208}, 1 (2007), 77--87.

\bibitem{DN}
{\sc Dokuchaev, M., and Novikov, B.}
\newblock Partial projective representations and partial actions.
\newblock {\em J. Pure Appl. Algebra 214}, 3 (2010), 251--268.

\bibitem{DN2}
{\sc Dokuchaev, M., and Novikov, B.}
\newblock Partial projective representations and partial actions {II}.
\newblock {\em J. Pure Appl. Algebra 216}, 2 (2012), 438--455.

\bibitem{DoNoPi}
{\sc Dokuchaev, M., Novikov, B., and Pinedo, H.}
\newblock The partial {S}chur multiplier of a group.
\newblock {\em J. Algebra 392\/} (2013), 199--225.

\bibitem{E-2}
{\sc Exel, R.}
\newblock The {B}unce-{D}eddens algebras as crossed products by partial
  automorphisms.
\newblock {\em Bol. Soc. Brasil. Mat. (N.S.) 25}, 2 (1994), 173--179.

\bibitem{E-1}
{\sc Exel, R.}
\newblock Circle actions on {$C^*$}-algebras, partial automorphisms and
  generalized {P}imsner-{V}oiculescu exact sequences.
\newblock {\em J. Funct. Anal. 122}, 3 (1994), 361--401.

\bibitem{E-3}
{\sc Exel, R.}
\newblock Approximately finite {$C^*$}-algebras and partial automorphisms.
\newblock {\em Math. Scand. 77}, 2 (1995), 281--288.

\bibitem{E0}
{\sc Exel, R.}
\newblock Twisted partial actions: a classification of regular
  {$C^*$}-algebraic bundles.
\newblock {\em Proc. London Math. Soc. 74}, 3 (1997), 417--443.

\bibitem{E1}
{\sc Exel, R.}
\newblock Partial actions of groups and actions of inverse semigroups.
\newblock {\em Proc. Amer. Math. Soc. 126}, 12 (1998), 3481--3494.

\bibitem{E3}
{\sc Exel, R.}
\newblock Hecke algebras for protonormal subgroups.
\newblock {\em J. Algebra 320}, 5 (2008), 1771--1813.

\bibitem{ELQ}
{\sc Exel, R., Laca, M., and Quigg, J.}
\newblock Partial dynamical systems and {$C^*$}-algebras generated by partial
  isometries.
\newblock {\em J. Operator Theory 47}, 1 (2002), 169--186.

\bibitem{EV}
{\sc Exel, R., and Vieira, F.}
\newblock Actions of inverse semigroups arising from partial actions of groups.
\newblock {\em J. Math. Anal. Appl. 363}, 1 (2010), 86--96.

\bibitem{F2}
{\sc Ferrero, M.}
\newblock Partial actions of groups on algebras, a survey.
\newblock {\em S{\~a}o Paulo J. Math. Sci. 3}, 1 (2009), 95--107.

\bibitem{GR}
{\sc Gon{\c c}alves, D., and Royer, D.}
\newblock Leavitt path algebras as partial skew group rings.
\newblock {\em Comm. Algebra 42}, 8 (2014), 3578--3592.

\bibitem{Howie}
{\sc Howie, J.~M.}
\newblock {\em An introduction to semigroup theory}.
\newblock Academic press, London-New York-San Francisco, 1976.

\bibitem{KL}
{\sc Kellendonk, J., and Lawson, M.~V.}
\newblock Partial actions of groups.
\newblock {\em Internat. J. Algebra Comput. 14}, 1 (2004), 87--114.

\bibitem{Lausch}
{\sc Lausch, H.}
\newblock Cohomology of inverse semigroups.
\newblock {\em J. Algebra 35\/} (1975), 273--303.

\bibitem{Lawson}
{\sc Lawson, M.~V.}
\newblock {\em Inverse semigroups. {T}he theory of partial symmetries}.
\newblock World Scientific, Singapore-New Jersey-London-Hong Kong, 1998.

\bibitem{Lawson2002}
{\sc Lawson, M.~V.}
\newblock {$E^*$}-unitary inverse semigroups.
\newblock In {\em Semigroups, algorithms, automata and languages ({C}oimbra,
  2001)}. World Sci. Publ., River Edge, NJ, 2002, pp.~195--214.

\bibitem{Leech75}
{\sc Leech, J.}
\newblock {$\cal H$}-coextensions of monoids.
\newblock {\em Mem. Amer. Math. Soc. 1}, 157 (1975), 1--66.

\bibitem{Loganathan81}
{\sc Loganathan, M.}
\newblock Cohomology of inverse semigroups.
\newblock {\em J. Algebra 70\/} (1981), 375--393.

\bibitem{NP}
{\sc Novikov, B., and Pinedo, H.}
\newblock On components of the partial {S}chur multiplier.
\newblock {\em Comm. Algebra 42}, 6 (2014), 2484--2495.

\bibitem{Pi}
{\sc Pinedo, H.}
\newblock On elementary domains of partial projective representations of
  groups.
\newblock {\em Algebra Discrete Math. 15}, 1 (2013), 63--82.

\bibitem{QR}
{\sc Quigg, J.~C., and Raeburn, I.}
\newblock Characterizations of crossed products by partial actions.
\newblock {\em J. Operator Theory 37}, 2 (1997), 311--340.

\bibitem{Sieben}
{\sc Sieben, N.}
\newblock {$C^*$}-crossed products by partial actions and actions of inverse
  semigroups.
\newblock {\em J. Aust. Math. Soc., Ser. A 63}, 1 (1997), 32--46.

\bibitem{S}
{\sc Steinberg, B.}
\newblock Partial actions of groups on cell complexes.
\newblock {\em Monatsh. Math. 138}, 2 (2003), 159--170.

\bibitem{Szendrei89}
{\sc Szendrei, M.~B.}
\newblock A note on {B}irget-{R}hodes expansion of groups.
\newblock {\em J. Pure Appl. Algebra 58}, 1 (1989), 93--99.

\end{thebibliography}


\begin{thebibliography}{1}

\bibitem{DK}
{\sc Dokuchaev, M., and Khrypchenko, M.}
\newblock {Partial cohomology of groups}.
\newblock {\em J. Algebra 427\/} (2015), 142--182.

\bibitem{E1}
{\sc Exel, R.}
\newblock {Partial actions of groups and actions of inverse semigroups}.
\newblock {\em Proc. Amer. Math. Soc. 126}, 12 (1998), 3481--3494.

\bibitem{KL}
{\sc Kellendonk, J., and Lawson, M.~V.}
\newblock {Partial actions of groups}.
\newblock {\em Internat. J. Algebra Comput. 14}, 1 (2004), 87--114.

\bibitem{Lawson}
{\sc Lawson, M.~V.}
\newblock {\em {Inverse semigroups. {T}he theory of partial symmetries}}.
\newblock World Scientific, Singapore-New Jersey-London-Hong Kong, 1998.

\bibitem{Lawson2002}
{\sc Lawson, M.~V.}
\newblock {{$E^*$}-unitary inverse semigroups}.
\newblock In {\em {Semigroups, algorithms, automata and languages ({C}oimbra,
  2001)}}. World Sci. Publ., River Edge, NJ, 2002, pp.~195--214.

\end{thebibliography}
\bibliographystyle{acm}

\end{document}

\title[Corrigendum to ``Partial cohomology of groups'']{Corrigendum  to ``Partial cohomology of groups''}

\author{Mikhailo Dokuchaev}
\address{Instituto de Matem\'atica e Estat\'istica, Universidade de S\~ao Paulo,  Rua do Mat\~ao, 1010, S\~ao Paulo, SP,  CEP: 05508--090, Brazil}
\email{dokucha@gmail.com}

\author{Mykola Khrypchenko}
\address{Departamento de Matem\'atica, Universidade Federal de Santa Catarina, Campus Reitor Jo\~ao David Ferreira Lima, Florian\'opolis, SC,  CEP: 88040--900, Brazil}
\email{nskhripchenko@gmail.com}

\subjclass[2010]{Primary 20M18; Secondary 20M15.}
\keywords{max-generated $F$-inverse monoid, minimum group congruence, Exel's monoid, epimorphism}

%\thanks{This work was partially supported by CNPq of Brazil and  FAPESP of Brazil (Proc. 2012/01554-7)}

\begin{abstract}
 We prove a corrected version of \cite[Theorem 5.4]{DK}.
\end{abstract}
\maketitle

Our aim  is to correct an inaccuracy in \cite[Lemma 5.3]{DK} which resulted in a flaw in \cite[Theorem 5.4]{DK}. To this end, let us introduce some notations, most of which were used in~\cite{DK} (for the inverse semigroup terminology see the monograph~\cite{Lawson}). Given an inverse semigroup $S$, we denote by $\s_S$ the minimum group congruence on $S$ and by $\cG S$ the maximum group image $S/\s_S$ of $S$. The $\s_S$-class of an element $s\in S$ (seen as a subset of $S$) will be denoted by $\s_S(s)$, and $\s_S^\nat$ will stand for the natural epimorphism $S\to\cG S$. We also recall that the Exel's monoid~\cite{E1} of a group $G$ is the universal semigroup $\cS G$ generated by the symbols $[g]$, $g\in G$, modulo the relations
	\begin{enumerate}
		\item $[g\m][g][h]=[g\m][gh]$;
		\item $[g][h][h\m]=[gh][h\m]$;
		\item $[g][1_G]=[g]$.
	\end{enumerate}
The elements $\e_g=[g][g\m]$ are commuting idempotents of $\cS G$, and each $s\in\cS G$ can be represented as $\e_{h_1}\dots\e_{h_n}[g]$, where $n\ge 0$, $h_i\ne h_j$ for $i\ne j$ and $h_i\not\in \{1_G,g\}$ for all $i$. Moreover, such a representation of $s$ is unique up to a permutation of the idempotents $\e_{h_i}$ (see~\cite[Propositions~2.5 and~3.2]{E1}). It is well-known~\cite{KL} that $\cS G$ is a max-generated~\cite{Lawson2002} $F$-inverse monoid, where $\max\s_{\cS G}(\e_{h_1}\dots\e_{h_n}[g])=[g]$ and $\cG{\cS G}\cong G$.

First, we would like to make some comments on~\cite[Lemma~5.1]{DK}, whose statement we reproduce here with some slight modification in notations and the specification of $\tl\pi$ in \cref{tilde-pi-is-iso}.

\begin{lem}[Lemma~5.1 from \cite{DK}]\label{ker-pi-in-sigma<=>tl-pi-iso}
	For an epimorphism $\pi:S\to T$ of inverse semigroups the following are equivalent:
	\begin{enumerate}
		\item $\ker\pi\sst\s_S$;
		\item $\tl\pi:\cG S\to \cG T$ mapping $\s_S^\nat(s)$ to $\s_T^\nat(\pi(s))$ is an isomorphism.\label{tilde-pi-is-iso}
	\end{enumerate}
	Moreover, in this case 
	\begin{align}\label{pi(sg_S(s))=sg_T(pi(s))}
		\pi(\s_S(s))=\s_T(\pi(s))
	\end{align}
	for all $s\in S$.
\end{lem}

Observe that if $S$ and $T$ are $F$-inverse monoids and $\pi:S\to T$ is an epimorphism satisfying \cref{pi(sg_S(s))=sg_T(pi(s))}, then
\begin{align}\label{pi(max-sg(s))=max-sg(pi(s))}
	\pi(\max\s_S(s))=\max\s_T(\pi(s)).
\end{align}
Indeed, $\pi$, being a homomorphism, respects the natural partial orders on $S$ and $T$, and since by \cref{pi(sg_S(s))=sg_T(pi(s))} any $t\in\s_T(\pi(s))$ is of the form $\pi(s')$ for some $s'\in\s_S(s)$, we have that $t=\pi(s')\le\pi(\max\s_S(s))$.

The following lemma, proved in~\cite{DK}, had a mistake in the ``uniqueness'' part of its statement.

\begin{lem}[Lemma~5.3 from \cite{DK}]\label{max-generated-as-image-of-S(G)}
	Let $S$ be a max-generated $F$-inverse monoid. Then there is a (unique up to an isomorphism) group $G$ with a (unique) epimorphism $\pi:\cS G\to S$, such that $\ker\pi\subseteq\sigma_{\cS G}$.
\end{lem}

Strictly speaking, the epimorphism $\pi$ is not unique, even if $G$ is fixed. Indeed, the map $\pi'=\pi\circ\phi$, where $\phi$ is an automorphism of $\cS G$, is also an epimorphism $\cS G\to S$ and $\ker{\pi'}\sst\s_{\cS G}$. We are going to show that these are in fact all epimorphisms $\pi':\cS G\to S$ with $\ker\pi'\subseteq\sigma_{\cS G}$. To this end, we first give a description of the automorphisms of $\cS G$, which is also interesting in itself.

\begin{lem}\label{Aut-S(G)}
	Each $\phi\in\Aut{\cS G}$ is of the form 
	\begin{align}\label{phi([g])=[vf(g)]}
		\phi([g])=[\vf(g)]
	\end{align}
	for some $\vf\in\Aut G$.
\end{lem}
\begin{proof}
	Indeed, each $\phi\in\Aut{\cS G}$ induces $\tl\phi\in\Aut{\cG{\cS G}}$ as in \cref{ker-pi-in-sigma<=>tl-pi-iso}\cref{tilde-pi-is-iso}. Identifying $\cG{\cS G}$ with $G$ and using \cref{pi(sg_S(s))=sg_T(pi(s))}, we obtain $\vf\in\Aut G$, such that for every $g\in G$
	\begin{align*}
		\phi(\s_{\cS G}([g]))=\s_{\cS G}([\vf(g)]).
	\end{align*}
Then in view of \cref{pi(sg_S(s))=sg_T(pi(s)),pi(max-sg(s))=max-sg(pi(s))}
\begin{align*}
	\phi([g])=\phi(\max\s_{\cS G}([g]))=\max\s_{\cS G}(\phi([g]))=\max\phi(\s_{\cS G}([g]))=[\vf(g)],
\end{align*}
giving \cref{phi([g])=[vf(g)]}.

Conversely, for any $\vf\in\Aut G$ define $\phi$ on the generators $[g]$ of $\cS G$ by \cref{phi([g])=[vf(g)]}. Since
\begin{align*}
	\phi([g\m])\phi([g])\phi([h])=[\vf(g\m)][\vf(g)][\vf(h)]=[\vf(g)\m][\vf(gh)]=\phi([g\m])\phi([gh]),
\end{align*}
and similarly $\phi([g])\phi([h])\phi([h\m])=\phi([gh])\phi([h\m]),$ and $\phi([g])\phi([1_G])=\phi([g])$, the map $\phi$ extends to a homomorphism $\cS G\to\cS G$, which we denote by the same letter $\phi$. Obviously, $\phi$ is onto, as $[\vf(g)]$ runs through all the generators of $\cS G$. Now suppose that 
\begin{align*}
	\phi(\e_{y_1}\dots\e_{y_n}[x])=\phi(\e_{v_1}\dots\e_{v_m}[u]).
\end{align*}
By the definition of $\phi$ we have that
\begin{align*}
	\e_{\vf(y_1)}\dots\e_{\vf(y_n)}[\vf(x)]=\e_{\vf(v_1)}\dots\e_{\vf(v_m)}[\vf(u)].
\end{align*}
The latter means that $\vf(x)=\vf(u)$ and $\{\vf(y_1),\dots,\vf(y_n)\}=\{\vf(v_1),\dots,\vf(v_m)\}$. Since $\vf$ is a bijection, we conclude that $x=u$ and $\{y_1,\dots,y_n\}=\{v_1,\dots,v_m\}$. So, $\e_{y_1}\dots\e_{y_n}[x]=\e_{v_1}\dots\e_{v_m}[u]$, and thus $\phi$ is one-to-one.
\end{proof}

Recall from the proof of \cite[Lemma~5.3]{DK} that the group $G$ from \cref{max-generated-as-image-of-S(G)} can be chosen to be $\cG S$.

\begin{lem}\label{pi-is-not-unique}
	Let $S$ be a max-generated $F$-inverse monoid, $G=\cG S$ and $\pi_1,\pi_2:\cS G\to S$ epimorphisms, such that $\ker\pi_i\subseteq\sigma_{\cS G}$, $i=1,2$. Then there exists an automorphism $\phi$ of $\cS G$, such that $\pi_2=\pi_1\circ\phi$.
\end{lem}
\begin{proof}
	By \cref{ker-pi-in-sigma<=>tl-pi-iso}\cref{tilde-pi-is-iso} the maps $\tilde\pi_i:\cG{\cS G}\to G$, $\tl\pi_i(\s^\nat_{\cS G}(s))=\s^\nat_S(\pi_i(s))$, $i=1,2$, are isomorphisms. Let us identify $\cG{\cS G}$ with $G$ as in the proof of \cref{Aut-S(G)}. Then $\tl\pi_i$ induces $\vf_i\in\Aut G$, such that
	\begin{align}\label{pi_i(sg([g]))=vf_i(g)}
		\pi_i(\s_{\cS G}([g]))=\vf_i(g),
	\end{align}
	where $i=1,2$. Define $\vf=\vf\m_1\circ\vf_2\in\Aut G$ and let $\phi\in\Aut{\cS G}$ be given by \cref{phi([g])=[vf(g)]}. Then by \cref{pi(max-sg(s))=max-sg(pi(s)),pi_i(sg([g]))=vf_i(g)}
	\begin{align*}
		(\pi_1\circ\phi)([g])&=\pi_1([\vf(g)])=\pi_1(\max\s_{\cS G}([\vf(g)]))\\
		&=\max\pi_1(\s_{\cS G}([\vf(g)]))=\max(\vf_1\circ\vf)(g)\\
		&=\max\vf_2(g)=\max\pi_2(\s_{\cS G}([g]))\\
		&=\pi_2(\max\s_{\cS G}([g]))=\pi_2([g])
	\end{align*}
	for all $g\in G$. Since the elements $[g]$, $g\in G$, generate $\cS G$, we conclude that $\pi_1\circ\phi=\pi_2$.
\end{proof}

Now~\cite[Corollary~5.2]{DK} together with the ``existence'' part of \cref{max-generated-as-image-of-S(G)} and \cref{pi-is-not-unique} yield the next.

\begin{thrm}[Corrected version of Theorem 5.4 from \cite{DK}]
	For an inverse semigroup $S$ the following are equivalent:
	\begin{enumerate}
		\item there exists a (unique up to an isomorphism) group $G$ admitting a (unique up to an automorphism of $\cS G$) epimorphism $\pi:\cS G\to S$ with $\ker\pi\subseteq\sigma_{\cS G}$;
		\item $S$ is a max-generated $F$-inverse monoid.
	\end{enumerate}
\end{thrm}

\section*{Acknowledgements}
The first author was partially supported by CNPq of Brazil and partially by FAPESP of Brazil. The second author was supported by FAPESP of Brazil (Proc. 2012/01554-7).

\bibliography{bibl-pact}{}
\bibliographystyle{acm}